\documentclass[11pt]{article} 
\usepackage{latexsym}
\usepackage{amsmath,amsthm,amsfonts,amssymb}
\usepackage[T1]{fontenc}
\usepackage[utf8]{inputenc}
\usepackage{aeguill}
\usepackage{url}
\usepackage{enumerate}
\usepackage{mdwlist} 
\usepackage{color}
\usepackage[a4paper,textwidth=14.8cm,textheight=24.4cm]{geometry}
\usepackage{xspace}
\usepackage{export}
\usepackage{hyperref}
\usepackage{pinlabel} 
\usepackage{stackrel} 
\usepackage{graphicx}

\usepackage[margin=1.5cm,small]{caption} 

\newtheorem{thm}{Theorem}[section]

\newtheorem*{thm*}{Theorem}

\newtheorem{cor}[thm]{Corollary}
\newtheorem*{cor*}{Corollary}

\newtheorem{prop}[thm]{Proposition}
\newtheorem*{prop*}{Proposition}
\newtheorem*{properties*}{Properties}

\newtheorem{lem}[thm]{Lemma}
\newtheorem*{lem*}{Lemma}

\newtheorem*{claim*}{Claim}

\newtheorem*{fact*}{Fact}
\newtheorem{fact}[thm]{Fact}

\newtheorem*{fait*}{Fait}

\newtheorem*{qst*}{Question}

\newtheorem*{pb*}{Problem}

\newtheorem*{conj*}{Conjecture}

\newtheorem*{exo*}{Exercise}

 \newtheorem{thmspecial}{Theorem}
 \newtheorem{corspecial}[thmspecial]{Corollary}

 \theoremstyle{definition}
 \newtheorem{dfn}[thm]{Definition}
 \newtheorem*{dfn*}{Definition}

  \newtheorem{dfnspecial}[thmspecial]{Definition}

\theoremstyle{remark}

\newtheorem*{algo*}{Algorithm}
\newtheorem*{rem*}{Remark}
\newtheorem{rem}[thm]{Remark}
\newtheorem*{example*}{Example}

\newcounter{numEnonceTmpInterne}
\newenvironment{enonce*}[1]{\theoremstyle{plain}\stepcounter{numEnonceTmpInterne}%
\def\a{enoncetmp\alph{numEnonceTmpInterne}}%
\newtheorem*{\a}{#1}\begin{\a}}{\end{\a}}

\makeatletter
\edef\@tempa#1#2{\def#1{\mathaccent\string"\noexpand\accentclass@#2 }}
\@tempa\rond{017}
\makeatother

\newcommand{\es}{\emptyset}
\renewcommand{\phi}{\varphi}
\newcommand{\m} {^{-1}}

\newcommand {\ra} {\rightarrow}

\newcommand {\onto} {\twoheadrightarrow}

\newcommand {\xra} {\xrightarrow}
\newcommand{\imp} {\Rightarrow}

\newcommand{\actson}{\curvearrowright}

\newcommand{\semidirect}{\ltimes}

\newcommand{\normal} {\vartriangleleft}

\newcommand{\ie} {i.~e.\ }

\newcommand {\cala} {{\mathcal {A}}}

\newcommand {\calc} {{\mathcal {C}}}

\newcommand {\calg} {{\mathcal {G}}}

\newcommand {\calm} {{\mathcal {M}}}

\newcommand {\calo} {{\mathcal {O}}}
\newcommand {\calp} {{\mathcal {P}}}
\newcommand {\calq} {{\mathcal {Q}}}

\newcommand {\cals} {{\mathcal {S}}}
\newcommand {\calt} {{\mathcal {T}}}
\newcommand {\calu} {{\mathcal {U}}}
\newcommand {\calv} {{\mathcal {V}}}

\newcommand {\bbF} {{\mathbb {F}}}

\newcommand {\bbQ} {{\mathbb {Q}}}
\newcommand {\bbR} {{\mathbb {R}}}

\newcommand {\bbZ} {{\mathbb {Z}}}

 \usepackage{bbm}

\newcommand{\abs}[1]{\lvert#1\rvert} 
\newcommand{\grp}[1]{\langle #1 \rangle}

\newcommand{\Out} {{\mathrm{Out}}}
\newcommand{\QH} {{\mathrm{QH}}}
\newcommand{\Hom} {{\mathrm{Hom}}}
\newcommand{\Aut} {{\mathrm{Aut}}}

\newcommand{\Inn} {{\mathrm{Inn}}}

\newcommand{\id} {\mathrm{id}}
\newcommand{\ad} {{\mathrm{ad}}}

\usepackage{ stmaryrd }

\newcommand{\Res}{\mathrm{Res}}
\newcommand{\Fact}{\mathrm{Fact}}

\newcommand{\gobble}[1]{} 

    \renewcommand{\>}{\rangle}

\newcommand{\lk}{\mathrm{lk}}
\newcommand{\st}{\mathrm{st}}   
\newcommand{\mids}{\, | \, }

\newcommand{\Mc}{\operatorname{Mc}}
\newcommand{\GL}{\operatorname{GL}}
\newcommand{\PGL}{\operatorname{PGL}}
\newcommand{\SL}{\operatorname{SL}}
\newcommand{\IA}{\operatorname{IA}}
\newcommand{\EH}{\operatorname{EH}}
\newcommand{\MCG}{\operatorname{MCG}}
\newcommand{\depth}{\operatorname{depth}}

\renewcommand{\vec}{\mathbf{v}}

\newcommand{\card}[1]{\#{#1}}
\newcommand{\A}{\mathbb{A}}

\usepackage{srcltx}

\begin{document}

\title{Vastness properties of automorphism groups of RAAGs}
\author{Vincent Guirardel, Andrew Sale}
\date{\today}

\maketitle

\begin{abstract}
Outer automorphism groups of RAAGs, denoted $\Out(\A_\Gamma)$, interpolate between $\Out(\bbF_n)$ and $\GL_n(\bbZ)$.
We consider several \emph{vastness} properties for  which $\Out(\bbF_n)$ behaves very differently from $\GL_n(\bbZ)$:
virtually mapping onto all finite groups,
SQ-universality, virtually having an infinite dimensional space of homogeneous quasimorphisms, and not being boundedly generated.

We give a neccessary and sufficient condition in terms of the defining graph $\Gamma$ for each of these properties to hold.
Notably, the condition for all four properties is the same, meaning $\Out(\A_\Gamma)$ will either satisfy all four, or none.
In proving this result, we describe conditions on $\Gamma$ that imply $\Out(\A_\Gamma)$ is large.

Techniques used in this work are then applied to the case of McCool groups, defined as subgroups of $\Out(\bbF_n)$ that preserve a given family of conjugacy classes.
In particular we show that any McCool group that is not virtually abelian virtually maps onto all finite groups, is SQ-universal, is not boundedly generated, and has a finite index subgroup whose space of homogeneous quasimorphisms is infinite dimensional. \end{abstract}

Right-angled Artin groups (or RAAGs, and also known as partially commutative groups, or graph groups) are
an important class of groups, containing free groups and free abelian groups as their extreme examples.
These groups are defined by a finite presentation whose set of relations is a subset of the set of commutation relations among generators.
Such a group $\A_\Gamma$ is conveniently defined by a finite graph $\Gamma$ whose vertex set is the set of generators of the presentation,
and where there is an edge between two generators when their commutator is a relation.

The fundamental group of a special cube complex embeds into a right-angled Artin group, giving them a rich and important collection of subgroups,
thanks to theory of Haglund and Wise \cite{Haglund-Wise-special}.

In this paper, it is not  the right-angled Artin groups themselves that we study, but rather their outer automorphism groups.
Free groups and free abelian groups being the extreme examples of right-angled Artin groups,
these automorphism groups \emph{interpolate} in some sense between the groups $\Out(\bbF_n)$ and  $\GL_n(\bbZ)$.

In many senses, $\GL_n(\bbZ)$ and $\Out(\bbF_n)$ exhibit very different types of behaviour (for $n\geq 3$).
For instance, $\GL_n(\bbZ)$ is not far from being simple for $n\geq 3$ since, except for $\PGL_n(\bbZ)$,
all its quotients are finite by Margulis' normal subgroup theorem \cite{Margulis_discrete}.
At the other extreme, $\Out(\bbF_n)$ is \emph{SQ-universal} for $n\geq 2$: any countable group embeds in some quotient of $\Out(\bbF_n)$.
Being virtually free, $\GL_2(\bbZ)$ is also SQ-universal.

There are many other senses in which $\Out(\bbF_n)$ is \emph{vaster} than $\GL_n(\bbZ)$.
Say that a finitely generated group $G$ \emph{involves all finite groups} if for every finite group $F$
there exists a finite index subgroup of $G$ that maps onto $F$.
For example, non-abelian free groups involve all finite groups, but virtually polycyclic groups, or more generally groups satisfying a law, do not.
It is well known (thanks to the congruence subgroup property) that for $n\geq 3$,
$\GL_n(\bbZ)$ does not involve all finite groups (see for instance \cite[Ch.{} 4]{LoRe_surface}).
On the other hand, $\GL_2(\bbZ)$ does involve all finite groups, since it is virtually free, as do
 mapping class groups and
$\Out(\bbF_n)$ for $n\geq 2$ \cite{MaRe_finite,GruLub_linear}.

Another such vastness property is related to bounded cohomology. Say that a finitely generated group $G$ has \emph{many quasimorphisms}
if the set of homogeneous quasimorphisms on $G$ is infinite dimensional.
If $G$ is finitely generated, then this implies that its bounded cohomology space $H^2_b(G;\bbR)$ is infinite dimensional.
It is known that every homogeneous quasimorphism on $\GL_n(\bbZ)$ is trivial,  for $n\geq 3$  \cite{Newman_unimodular,BuMo_cohom_lattice}.
Meanwhile $\SL_2(\bbZ)$ and $\Out(\bbF_n)$ have many quasimorphisms, for each $n\geq 2$ \cite{BF_hyperbolic}.

We can also consider bounded generation: one says that $G$ is \emph{boundedly generated} if there exist $g_1,\ldots, g_k \in G$ such that any $g\in G$ can be written as $g=g_1^{\alpha_1}\ldots g_k^{\alpha_k}$ for some $\alpha_1,\ldots,\alpha_k \in \bbZ$.
It is known since Carter--Keller \cite{CarterKeller} that $\GL_n(\bbZ)$ is boundedly generated.
At the other extreme, it follows from \cite{BF_hyperbolic} that $\Out(\bbF_n)$ is not boundedly generated.

Since the groups $\Out(\A_\Gamma)$ of outer automorphisms of RAAGs interpolate between $\GL_n(\bbZ)$ and $\Out(\bbF_n$),
for each of these properties it is natural to ask what is the ``boundary'' (in the space of all finite graphs $\Gamma$)
between the vast behaviour, and the skimpy one.

Our main result shows that the four properties above give rise to the same boundary, and we describe it precisely in terms of the graph $\Gamma$.

A stronger property of groups that implies each of the above vastness properties is that of being \emph{large}.
A group $G$ is said to be large if it has a finite index subgroup that maps onto a non-abelian free group.
The question of whether $\Out(\bbF_n)$ is large is still open for $n>3$, with it known to be large for $n=2,3$ \cite{GruLub_linear}.
Part of our main result describes properties of $\Gamma$ that imply that $\Out(\A_\Gamma)$ is large.

Before stating our theorems,
let us first recall  a result by M.{} Day saying
which graphs $\Gamma$ are such that $\Out(\A_\Gamma)$ contains a free group \cite{Day_solvable}.
This involves the following two features of a graph that were introduced by Servatius \cite{Servatius} and by Gutierrez--Piggott--Ruane \cite{GPR_automorphisms}.

The first feature is a preordering $\leq$ on the vertex set of $\Gamma$ defined by
$u\leq v$ if $\lk(u)\subset \lk(v)\cup\{v\}$
(where $\lk(u)$ is the set of vertices of $\Gamma\setminus\{u\}$ joined by an edge to $u$;
 see Section \ref{sec:preordering} for more on the preordering).
Servatius introduced this notion (though not in this language) as it is the right condition that allows the existence of \emph{transvection} automorphisms
$R_u^v, L_u^v$ sending $u$ to $uv$ or $vu$ respectively, and fixing all the other generators.
Associated to this preordering is an equivalence relation: $u\sim v$ if $u\leq v$ and $v\leq u$.
Note that the equivalence classes determined by the preordering  are either \emph{abelian},
if the equivalence class forms a clique in $\Gamma$, or \emph{free}, if there is no edge between any pair of vertices in the equivalence class
(it is both free and abelian if it contains a single vertex).

If there exist two distinct vertices $u$ and $v$ such that $u\sim v$, the transvections $R_u^v$, $R_v^u$, $L_u^v$, $L_v^u$ exist in $\Out(\A_\Gamma)$,
and one easily checks that the subgroup of $\Out(\A_\Gamma)$ they generate contains a free group:  according to whether $u$ commutes with $v$ or not,
it naturally maps onto to $\SL_2(\bbZ)$ or to $\Out^+(\bbF_2)$ (these two groups being in fact isomorphic),
where $\Out^+(\bbF_n)$ is the index 2 subgroup of $\Out(\bbF_n)$
containing the elements whose action on the abelianisation has determinant $1$.

The second feature
is called a \emph{separating intersection of links} (SIL).
It was introduced by Gutierrez--Piggott--Ruane \cite{GPR_automorphisms}, who showed that the subgroup generated by partial conjugations is abelian if and only if $\Gamma$ has no SIL (in fact their results concern more general graph products).
In this paper, a SIL is a triple $(x,y\mids z)$ of pairwise non-commuting vertices $x,y,z \in \Gamma$, such that
the connected component $Z$ of $\Gamma\setminus \lk(x)\cap \lk(y)$ containing $z$ contains neither $x$ nor $y$ (see Section \ref{subsec_SIL} for more details).
This feature allows one to define two \emph{partial conjugations} $C_Z^x$, $C_Z^y$ that act as the identity on the generators in $\Gamma\setminus Z$,
and conjugates all the generators in $Z$ by $x$ (resp. by $y$).
By a normal form argument, one easily checks that these two partial conjugations generate a free group in $\Out(\A_\Gamma)$.

In \cite{Day_solvable}, Day shows that these two sources suffice to explain the existence of free subgroups:
$\Out(\A_\Gamma)$ contains a free group if and only if it contains an equivalence class of size at least 2 or a SIL;
 moreover, Day proves that otherwise, $\Out(\A_\Gamma)$ is virtually nilpotent.\\

We are now ready to state our main result.

\begin{thmspecial}\label{thm_various}
  Given a finite graph $\Gamma$, the following are equivalent:
\begin{itemize}
\item $\displaystyle (*)\begin{cases}
\text{$\Gamma$ has a SIL,}\\
\text{or a non-abelian equivalence class,}\\
\text{or an abelian equivalence class of size 2;}
\end{cases}
$
\item $\Out(\A_\Gamma)$ is SQ-universal;
\item $\Out(\A_\Gamma)$ involves all finite groups;
\item $\Out(\A_\Gamma)$ virtually has many quasimorphisms;
\item $\Out(\A_\Gamma)$ is not boundedly generated.
\end{itemize}
\end{thmspecial}

Note that there is no obvious
general relation between the four group theoretic properties  in Theorem \ref{thm_various},  except that
boundedly generated groups don't have many quasimorphisms.
For example, the free product of two infinite simple groups is SQ-universal, but does not involve every finite group.
Hence, the fact that these four properties share the same boundary looks like a coincidence. It can be explained thanks to the following
trichotomy concerning $\Out(\A_\Gamma)$.

\begin{thmspecial}\label{thm_tricho}
  Let $\Gamma$ be a finite graph. Then  there is a finite index subgroup $\Out^0(\A_\Gamma)$ of $\Out(\A_\Gamma)$ such that:
  \begin{itemize}
  \item[$\mathrm{(1a).}$] if $\Gamma$ has a non-abelian  free equivalence class  of size $n\geq 2$,
   or an abelian equivalence class of size 2,
then $\Out^0(\A_\Gamma)$ maps onto $\Out^+(\bbF_n)$ or $\SL_2(\bbZ)$ accordingly;
\item[$\mathrm{(1b).}$] if all equivalence classes of $\Gamma$ are abelian, and $\Gamma$ has a SIL, then $\Out(\A_\Gamma)$ is large;
\item[$\mathrm{(2).}$] if $\Gamma$ has no SIL,  and all its equivalence classes are abelian of sizes $n_1,\dots,n_k\neq 2$,
then $\Out^0(\A_\Gamma)$ fits in a short exact sequence
$$1\ra P\ra \Out^0(\A_\Gamma) \ra \prod_{i=1}^k \SL_{n_i}(\bbZ)\ra 1$$
where $P$ is finitely generated and nilpotent.
  \end{itemize}
\end{thmspecial}

The finite index subgroup $\Out^0(\A_\Gamma)$ is the subgroup generated by transvections and partial conjugations (see Definition \ref{dfn_out0}).
Slightly enlarging $\Out^0(\A_\Gamma)$ (by including inversions)
will instead yield a map onto $\Out(\bbF_n)$ or $\GL_2(\bbZ)$ in assertion (1a) above. See Remark \ref{rem:add inversions}.

We note that assertions (1a) or
(1b) hold if and only if $(*)$ holds, while assertion (2) holds if and only if $(*)$ does not hold.

In fact, assertions (1a) and (2) are particular cases of more general statements.
Proposition \ref{prop_tricho1} says that for any equivalence class $[u]$, there is a natural epimorphism
$\Out^0(\A_\Gamma)\onto \Out^0(\A_{[u]})$.
Proposition \ref{prop_tricho3} shows that the statement in assertion (2)
also holds under the sole assumption that $\Gamma$ has no SIL, in which case some $n_i$ may equal $2$. In fact, this assumption implies
there is no free equivalence class of size at least $3$ and one can prove that the free equivalence classes of size 2 have to correspond to direct products with free groups (see Lemma \ref{lem_product}).

Since $\Out(\bbF_2)\simeq GL_2(\bbZ)$ and $\Out(\bbF_3)$ are large \cite{GruLub_linear}, we deduce:

\begin{corspecial}
	For every finite graph $\Gamma$, either $\Out(\A_\Gamma)$ has a finite index subgroup that fits in a short exact sequence as in Theorem \ref{thm_tricho} (2), or $\Out(\A_\Gamma)$ is large, or $\Out(\A_\Gamma)$ has a finite index subgroup that surjects onto $\Out(\bbF_n)$ for some $n\geq 4$.
\end{corspecial}

Formally, Theorem 2 only gives a virtual map onto $\Out^+(\bbF_n)$, but as noted above, Remark \ref{rem:add inversions} yields a virtual map
to $\Out(\bbF_n)$.

To deduce Theorem \ref{thm_various} from Theorem \ref{thm_tricho},
we observe that the four vastness properties in Theorem \ref{thm_various}
hold for assertions (1a), (1b) of Theorem \ref{thm_tricho} because they do so for
  $\Out(\bbF_n)$, $\SL_2(\bbZ)$, and any large group.
On the other hand, if assertion (2) holds, then $\Out^0(\A_\Gamma)$ maps with  finitely generated nilpotent kernel
onto a product of $\SL_{n_i}(\bbZ)$ with $n_i\neq 2$;
since $\SL_{n_i}(\bbZ)$ does not satisfy these vastness properties,
one easily deduces that neither does $\Out^0(\A_\Gamma)$ (see Section \ref{sec:prelim vast} for more details).

One can in fact place these properties in an abstract framework.

\begin{dfnspecial}\label{dfn:vastness}
We say a \emph{vastness property} of a group is a property $\calp$ that satisfies the following conditions:
\begin{enumerate}
\renewcommand{\theenumi}{($\calv$\arabic{enumi})}
\item If $G\onto G'$ and $G'$ has $\calp$, then so does $G$.\label{vast:surjection}
\item Let $H$ be a finite index subgroup in a group $G$. Then $G$ has $\calp$ if and only if $H$ does.\label{vast:finite index}
\item Not having $\calp$ is stable under extensions:
if $1\ra N\ra G\ra Q\ra 1$ is an exact sequence where $N$ and $Q$ do not have $\calp$, then $G$ does not have $\calp$.\label{vast:SES}
\end{enumerate}
\end{dfnspecial}

SQ-universality, involving all finite groups, virtually having many quasimorphisms,  and  not being boundedly generated
are all vastness properties in this sense (see Section \ref{sec:prelim}), and they are not satisfied by $\bbZ$.
Note that if a vastness property is not satisfied by $\bbZ$, then it is not satisfied by any virtually polycyclic group,
including finitely generated nilpotent groups (being infinite, or virtually mapping onto $\bbZ$ being examples of vastness properties satisfied by $\bbZ$).
Thus, Theorem \ref{thm_various} can be viewed as a particular case of the following:

\begin{corspecial}\label{cor:vast out raag}
  Let $\calp$ be any vastness property that is satisfied by  $\Out(\bbF_n)$ for all $n\geq 2$, but not by $\bbZ$ or by
$\SL_n(\bbZ)$ for $n\geq 3$.

Then $\Out(\A_\Gamma)$ satisfies $\calp$ if and only if the graph $\Gamma$ satisfies $(*)$.
\end{corspecial}

Finally, we give conditions on the graph $\Gamma$ that imply the outer automorphism group to be large.

\begin{thmspecial}\label{thmspecial:raags_large}
	If $\Gamma$ has an equivalence class of size two (abelian or not),
	or a non-abelian equivalence class of size three,
	or if all  equivalences classes are abelian and there is a SIL,
	then $\Out(\A_\Gamma)$ is large.
\end{thmspecial}

In particular, under the hypotheses of Theorem \ref{thmspecial:raags_large}, $\Out(\A_\Gamma)$ does not have Kazhdan's  property (T);
see also the results of Aramayona and Martinez-Perez about the same question \cite{AM16_firstcohom}.

\medskip

In addition, we apply our techniques to study vastness properties of McCool groups.
A \emph{McCool group} $\Mc(\bbF_n;\calc)$ of the free group $\bbF_n$ is a subgroup of $\Out(\bbF_n)$
defined as the stabilizer  of a finite set $\calc$ of conjugacy classes in $\bbF_n$
(see \cite{GL_McCool} for seemingly more general but equivalent definitions, particularly Corollary 1.6 therein).
For this class of groups, we prove the following:

\begin{thmspecial}
Let $M=\Mc(\bbF_n;\calc)$ be a McCool group of $\bbF_n$.

If $M$ is not virtually abelian, then
	\begin{itemize}
		\item it is SQ-universal,
		\item it involve all finite groups,
		\item it virtually has many quasimorphisms,
		\item it isn't boundedly generated.
	\end{itemize}
\end{thmspecial}

This result follows from the following alternative:
\begin{thmspecial}\label{thmspecial:mccool}
	Let $\calc$ be a finite set of  conjugacy classes in $\bbF_n$. Then either
	\begin{itemize}
		\item $\Mc(\bbF_n;\calc)$ is large,
		\item $\Mc(\bbF_n;\calc)$ maps onto $\Out(\bbF_r)$,  where $\bbF_r$ is a free group of rank $r$ with  $2\leq r\leq  n$,
		\item $\Mc(\bbF_n;\calc)$ maps onto a  non-virtually abelian mapping class group of a (not necessarily orientable) hyperbolic surface with at least one puncture,
		\item $\Mc(\bbF_n;\calc)$ is virtually abelian.
	\end{itemize}
\end{thmspecial}

A more precise version of Theorem \ref{thmspecial:mccool} is given later in the text, as Theorem \ref{thm_MC}.
In particular, we explain how the nature of the Grushko decomposition of $\bbF_n$ relative to $\calc$ influences the possible outcomes.

A consequence of Theorem \ref{thmspecial:mccool} is the following:

\begin{corspecial}\label{corspecial:mccool}
	Let $\calp$ be any vastness property that holds for $\Out(\bbF_n)$, for $n\geq 2$, and for all non-virtually abelian mapping class groups of punctured hyperbolic surfaces.
	
	Then any McCool group of $\bbF$ is either virtually abelian or satisfies $\calp$.
\end{corspecial}	

In particular we show in Proposition \ref{prop:v abelian mcg} that all such mapping class groups are either virtually abelian or involve all finite groups (see \cite{MaRe_finite} for the orientable case, including for non-punctured surfaces). We also explain therein how \cite{BeFu_cohomology,BeFu_quasi,DGO_HE} imply the other three vastness properties also hold. Hence Corollary \ref{corspecial:mccool} leads to the following.

\begin{corspecial}
	Any  McCool group of a free group $\bbF$ which is not virtually abelian involves all finite groups, is SQ-universal, virtually has many quasimorphisms, and is not boundedly generated.
\end{corspecial}

 We now outline the structure of the paper.
The necessary preliminary materials are given in Section \ref{sec:prelim}, before giving way to the proof of Theorem \ref{thm_tricho} in Sections \ref{sec:free equivs}--\ref{sec:SES when no SIL}.
The proofs of the alternatives of the trichotomy are essentially independent and occupy a section each.

In Section \ref{sec:free equivs}, we make the observation  that for any equivalence class $[v]\subset \Gamma$ of size $n$,
$\Out^0(\A_\Gamma)$ always surjects onto the group $\Out^0(\A_{[v]})$, which is either isomorphic to  $\Out^+(\bbF_n)$ or to $\SL_n(\bbZ)$,
according to whether $[v]$ is abelian or not.
Although not difficult, this observation seems to be new as we did not find it in the literature (see \cite{CV09} when $[v]$ is maximal).

In Section  \ref{sec:SIL}, we take care of assertion (1b) of the trichotomy, showing  that if all equivalence classes are abelian, then the existence of a SIL implies that $\Out(\A_\Gamma)$ is large. This is proved by looking for a  SIL $(x_1,x_2\mids x_3)$ that is in some sense minimal, which yields a homomorphism
of $\Out^0(\A_\Gamma)$ onto a subgroup $\calo$ of $\Out(\bbZ^{n_1}*\bbZ^{n_2}*\bbZ^{n_3})$  that has a well-understood set of generators
(where $n_i$ is the size of the equivalence class $[x_i]$).
To prove that the subgroup $\calo$ is large,
we use a linear representation of $\Out(\bbZ^{n_1}*\bbZ^{n_2}*\bbZ^{n_3})$
coming from the  action on the first rational homology of an index 2 subgroup of $\bbZ^{n_1}*\bbZ^{n_2}*\bbZ^{n_3}$,
similar to the representations constructed by Looijenga and Grunewald--Larsen--Lubotzky--Malestein for mapping class groups, and by Grunewald--Lubotzky for $\Out(\bbF_n)$ \cite{Looi97,GLLM_arithmetic_mcg,GruLub_linear}.
In some cases, the image of $\calo$ in this representation is too small to imply largeness,
but we are instead able to show that $\calo$ splits as an amalgam or an HNN extension that allows us to deduce largeness.

In Section \ref{sec:SES when no SIL}, we cover the last case of the trichotomy.
We use that when there is no SIL the group $\IA(\A_\Gamma)$ is abelian, where $\IA(\A_\Gamma)\subset \Out(\A_\Gamma)$ is
defined as the kernel of the map $\Out(\A_\Gamma)\ra GL_{\card\Gamma}(\bbZ)$ induced by the abelianisation
(we give a short proof, see also \cite{CRSV_no_SIL,GPR_automorphisms}).
This uses a generating set of $\IA(\A_\Gamma)$ given by Day \cite{Day_symplectic}, see also \cite{Wade_Phd}.
Since the image of $\Out^0(\A_\Gamma)$ in $\GL_{\card\Gamma}(\bbZ)$ is block triangular with diagonal blocks of size $n_i$, where
the integers $n_i$ are the sizes of the equivalence classes in $\Gamma$, we get a short exact sequence as in $\mathrm{(2)}$ with $P$ polycyclic. An argument by Day allows to show that $P$ is in fact nilpotent.

Finally, Section \ref{sec:mccool} deals with McCool groups.
One aspect we need to deal with here is to show that the vastness properties hold for mapping class groups that are not virtually abelian, and the surface is punctured and hyperbolic, and may not be orientable.
To show all finite groups are involved, we use the Birman exact sequence \cite{Birman_MCG,Korkmaz_MCG} (see also \cite{MaRe_finite} for orientable surfaces, punctured or not).
The remaining three properties follow from acylindrical hyperbolicity.

\paragraph{Acknowledgements.} We would like to thank Ruth Charney, Denis Osin, Mark Sapir, and Karen Vogtmann for various conversations,  and the anonymous referees for their careful reading and helpful comments that allowed to simplify some of the arguments. The two authors acknowledge support from ANR grant ANR-11-BS01-013.
The first author acknowledges support from the Institut Universitaire de France, and would like to thank the Centre Henri Lebesgue ANR-11-LABX-0020 LEBESGUE. 

\section{Preliminaries}\label{sec:prelim}

The material we cover in the section is as follows:
after describing the pertinent features of RAAGs, Section \ref{sec:prelim vast} discusses the proposed vastness properties,
Section \ref{sec:prelim McCool} gives some basic properties of McCool groups, needed in Section \ref{sec:mccool},
and Section \ref{sec:prelim usable results} contains some useful properties of groups of outer automorphisms of both free products and HNN extensions over the trivial group.

\subsection{Right-angled Artin groups}

Given a simplicial graph $\Gamma$, we define the associated RAAG, denoted $\A_\Gamma$.
The generating set of $\A_\Gamma$ is the set of vertices of $\Gamma$.
A set of defining relators consists of commutators $[u,v]$ for all pairs $u,v$ of adjacent vertices in $\Gamma$.
That is, two generators commute if and only if they correspond to adjacent vertices.

Given a subset $V$ of vertices of $\Gamma$, we define $\A_V$ as the RAAG associated to the induced subgraph.
Since $\A_V$ naturally embeds into $\A_\Gamma$, we identify $\A_V$ with the subgroup of $\A_\Gamma$ generated by the vertices of $V$.

\subsubsection{The partial preordering}\label{sec:preordering}

We make use of a partial  preordering,  introduced  as \emph{domination} by Servatius \cite{Servatius}, and also described in \cite{CV09}.
Let $u,v$ be vertices in $\Gamma$. We say:
$$u \leq v \textrm{ whenever }\lk(u) \subseteq \st(v)$$
 where $\lk(u)$ is the \emph{link} of $u$, namely the set of all vertices connected to $u$ by a single edge, and $\st(v)=\lk(v)\cup\{v\}$ is the \emph{star} of $v$.
We note that if $u$ and $v$ are distinct, adjacent vertices
then $u\leq v$ is equivalent to $\st(u)\subseteq \st(v)$. If they are distinct, non-adjacent vertices
then $u \leq v$ is equivalent to $\lk(u)\subseteq \lk(v)$.

The partial  preordering defines equivalence classes: $u\sim v$ if $u\leq v$ and $v\leq u$.
We denote by $[u]$ the equivalence class of $u$.
If an equivalence class $[u]$ contains two adjacent vertices, then
all vertices in $[u]$ have the same star, so $[u]$ is a clique (any two vertices are adjacent).
Otherwise, all vertices in $[u]$ share the same link and generate a free group.
We therefore say that $[u]$ is \emph{abelian} if all its vertices are adjacent, (\ie if the  subgroup  generated by $[u]$ is abelian),
and we say that $[u]$ is \emph{free} if it contains no adjacent pair of vertices (\ie if the  subgroup generated by $[u]$ is free).
An equivalence class of size 1 is both abelian and free.
Since vertices of a free  equivalence class share the same link,
 we may think of this link as being the link of the equivalence class, and similarly for the star of an abelian equivalence class.
For more details on this partial preordering, we refer the reader to \cite{CV09}.

\subsubsection{Automorphisms}\label{sec:prelim-generators}

A finite set of generators for the (outer) automorphism group of a right-angled Artin group was given  by Servatius and Laurence \cite{Servatius,Laur95}.
	The generators may be viewed as automorphisms, however we will take them to represent their class in $\Out(\A_\Gamma)$, unless explicitly stated otherwise.
The generating set consists of the following elements:

\begin{itemize}
	\item \emph{involution} of a generator: sends one generator $v$ to its inverse $v ^{-1}$ and fixes all others;
	\item \emph{graph symmetry}: permute the generating set according to a symmetry of $\Gamma$;
	\item \emph{transvections}: if $v \leq w$ then define the left-transvection $L_v^w$  (respectively right-transvection $R_v^w$) by
	$$L_v^w(v)=wv\ \text{ \quad  (respectively $R_v^w(v)=vw$)}$$
	while fixing all other generators. The \emph{multiplier} of $L_v^w$ and $R_v^w$ is the vertex $w$.
	\item \emph{partial conjugations}: for $v \in V$, let $Z$ be a union of connected components of $\Gamma \setminus \st(v)$. Define the partial conjugation $C_Z^v$ by sending $z$ to $vzv^{-1}$ if $z \in Z$ and fixing $z$ otherwise. The \emph{multiplier} of the partial conjugation is the vertex $v$.
\end{itemize}

\begin{rem}	
	Often, one restricts partial conjugations to the case where $Z$ is connected.
	 However our definition agrees with the definition given by Laurence, where he used the name \emph{locally inner} for partial conjugations \cite{Laur95}.
	Clearly, our partial conjugations include these automorphisms, and also those that are products of partial conjugations in the restricted sense.
	It will be more convenient for us to use this more flexible notion, as then the inverse of a partial conjugation is again a partial conjugation (in $\Out(\A_\Gamma)$), and the restriction and factor maps defined below will
	send partial conjugations to partial conjugations.
\end{rem}
	
The inner automorphisms of $\Aut(\A_\Gamma)$ are denoted by $\ad_g : h \mapsto ghg\m$.

\begin{dfn}\label{dfn_out0}
	We will consider the subgroup $\Out^0(\A_\Gamma)$ of $\Out(\A_\Gamma)$ that is generated by transvections and partial conjugations.
\end{dfn}

The generating set of $\Out^0(\A_\Gamma)$ does not include the graph symmetries and involutions,
and it is not hard to see that $\Out^0(\A_\Gamma)$ is a normal subgroup of finite index in $\Out(\A_\Gamma)$.
For instance, for $\A_\Gamma=\bbZ^n$, $\Out^0(\A_\Gamma)=\SL_n(\bbZ)$, and
for $\A_\Gamma=\bbF_n$, $\Out^0(\A_\Gamma)=\Out^+(\bbF_n)$, the index 2 subgroup of $\Out(\bbF_n)$
consisting of outer automorphisms of $\bbF_n$ inducing an automorphism of $\bbZ^n$ with determinant $1$.

We note that in the literature, $\Out^0(\A_\Gamma)$ often also includes inversions.

\subsubsection{The standard representation of $\Out(\A_\Gamma)$}\label{sec:standard representation}

In the following, we describe a representation of  $\Out(\A_\Gamma)$, which we shall hereupon call the \emph{standard representation}.
It is obtained by having $\Aut(\A_\Gamma)$ act on the abelianisation of $\A_\Gamma$, and we denote it by  $\sigma:\Out(\A_\Gamma)\ra \GL_{\card{\Gamma}}(\bbZ)$.
Fix an enumeration $v_1,\dots,v_n$ of the vertices of $\Gamma$ that agrees with the partial preordering defined above:
vertices of an equivalence class have adjacent indices, and
if $v_i\leq v_j$ with $v_i\not\sim v_j$ then $i\leq j$.
Left and right transvections are mapped to elementary matrices in $\SL_{\card{\Gamma}}(\bbZ)$, and, since partial conjugations are in the kernel of $\sigma$, we see that the image of $\Out^0(\A_\Gamma)$ has a block lower-triangular form.

Indeed, define the algebraic group
$$\mathcal{G} = \big\{ (x_{u,v})_{u,v\in\Gamma} \in \SL_{\card{\Gamma}}(\bbR) \mid x_{u,v} = 0 \textrm{ if } v \nleq u\big\} .$$
With this notation, the image of $\Out^0(\A_\Gamma)$ under $\sigma$ coincides with $\mathcal{G}_\bbZ=\calg\cap \SL_{\card{\Gamma}}(\bbZ)$.
There is a diagonal block $\SL_d(\bbZ)$ for each equivalence class of size $d$, and an off-diagonal block of integer matrices $M_{d,d'}(\bbZ)$
for each pair of equivalence classes
$[v]\leq [v']$ of sizes $d,d'$.

\subsubsection{The Torelli group}\label{sec:torelli}

We denote by $\IA(\A_\Gamma)$ the kernel of the standard representation. This is sometimes called the Torelli group of $\A_\Gamma$.
Day \cite[\S 3]{Day_symplectic} showed that $\IA(\A_\Gamma)$ is finitely generated, describing explicitly a finite generating set $\mathcal{M}_\Gamma$. We also refer the reader to Wade \cite{Wade_Phd} who independently proved the same result.

The set $\mathcal{M}_\Gamma$ consists of all partial conjugations along with the commutator transvections $K_u^{[v,w]}$, whenever $u,v,w$ are all distinct,  $[v,w]\neq 1$, and $u \leq v,w$, defined as
$$K_u^{[v,w]}(y) = \left\{ \begin{array}{ll} u[v,w] & \textrm{if $ y=u$};\\ y & \textrm{otherwise.} \end{array}\right.$$

\subsection{Vastness properties}\label{sec:prelim vast}

Recall that we say a property of groups is a vastness property if it satisfies \ref{vast:surjection}, \ref{vast:finite index}, and \ref{vast:SES}, from the introduction.
In light of Corollary \ref{cor:vast out raag}, we are particularly interested in those vastness properties that are satisfied by $\Out(\bbF_n)$, for $n\geq 2$, but not by $\SL_n(\bbZ)$ for $n\geq 3$.

 \subsubsection{Involving all finite groups}

 \begin{dfn}
 	We say that a group $G$  \emph{involves all finite groups} if for every finite group $F$
 	there exists a finite index subgroup $G_0<G$ and an epimorphism $G_0\onto F$.
 \end{dfn}

 For instance, the free group $\bbF_2$  involves all finite groups.
 Clearly if a quotient of $G$, or a finite index subgroup of $G$ involves all finite groups, then so does $G$.
 In particular,   \ref{vast:surjection} holds, and one direction of \ref{vast:finite index} also holds.
 Note that if $G$ is large then it involves all finite groups.

 On the other hand, if $G$ is abelian, nilpotent, or virtually solvable (and more generally any group satisfying a law)
 then $G$ does not involve all finite groups,
 since some finite groups will not satisfy the law(s).

 \begin{rem}\label{rem:afgi An}
 	Note that all finite groups are involved in a group $G$ if and only if there exist infinitely many integers $n$ such that
 	 	there is a finite-index subgroup $G_n<G$ with an epimorphism of $G_n$ onto the alternating group $A_{n}$.
 	 	Indeed, any finite group $F$ embeds into a symmetric group $S_k$, for some $k$, which can then be embedded into $A_{k+2}$.
 Given an epimorphism $p:G \onto A_{k+2}$, $p\m(F)$ is a finite-index subgroup of $G$ that maps onto $F$.
 \end{rem}

 The following lemma gives the behaviour of this property under group extensions,  and is proved below.

 \begin{lem}\label{lem:technical_lemma}
 	Let $1\ra N\ra G\ra Q\ra 1$ be a short exact sequence.
 	\begin{enumerate}
 		\item If $G$ involves all finite groups, then so does either $N$ or $Q$.
 		\item If $N$ is finitely generated and involves all finite groups, then so does $G$.
 	\end{enumerate}
 \end{lem}

The first part is sufficient to imply property \ref{vast:SES}, thus leading to the following.

 \begin{cor}
 	Having all finite groups involved is a vastness property.
 \end{cor}

 \begin{proof}
 	Lemma \ref{lem:technical_lemma} implies assertion \ref{vast:SES} of vastness properties.
 	The only non-obvious property remaining to check is that if $G$ involves all finite groups,
 	then so does any finite index subgroup $H$.
 	Passing to a further finite index subgroup, we can assume that $H$ is normal in $G$.
 	Since $G/H$ is finite it cannot have all finite groups involved, so Lemma \ref{lem:technical_lemma} implies $H$ does.
\end{proof}

 \begin{proof}[Proof of Lemma \ref{lem:technical_lemma}.]
 	Assume that $G$ involves all finite groups, but $N$ does not.
 	This means, by Remark \ref{rem:afgi An}, there exists some $n_0$ such that for all $n\geq n_0$, there is no virtual homomorphism from $N$ onto $A_n$.
 	Let $n\geq n_0$ and $G_0$ be a finite index subgroup with an epimorphism  $p:G_0\onto A_{n}$.
 	Restrict the short exact sequence to
 	$$1\ra N_0\ra G_0\ra Q_0\ra 1.$$
 	The group $p(N_0)$  is a normal subgroup of $A_n$ and can't be $A_n$, so it must be trivial.
 	Hence $p$ factors through $Q_0$ and consequently, $Q$ involves all finite groups.
 	This proves the first part.
 	
 	For the second part, we want to show that $G$ virtually surjects onto $A_n$ for all $n$ large enough,
 	assuming this is true of $N$ and $N$ is finitely generated.
 	Take $N_0<N$ of finite index with an epimorphism $\phi:N_0\onto A_n$.
 	Let $G_1$ be the normalizer of $N_0$, so that $N_0\normal G_1$.
 	Then $G_1$ has finite index in $G$ since, being finitely generated, $N$  has only finitely many subgroups of a given index  and each coset of $G_1$ in $G$ determines a different subgroup of $N$ that is a conjugate of $N_0$.
 	
 	Consider the family of epimorphisms
 	$$\phi_g=\phi\circ \ad_g:N_0 \onto A_n , \ \ g\in G_1$$
 	 where $\ad_g$ is an inner automorphism of $G_1$.
 	Since $N_0$ is finitely generated, it has only finitely many homomorphisms to $A_n$, so
 	there are only finitely many (say $k$) distinct maps in this family.
 	Let $N_1=\cap_{g\in G_1} \ker \phi_g$, which is normal in $G_1$. Upon taking the quotient, we get
 	$N_0/N_1 \normal G_1/N_1$, and $N_0/N_1$ is a finite group that embeds into $(A_n)^k$ as a subgroup that surjects onto each factor.
 	In particular, $N_0/N_1$ has trivial centre.
 	Now let $G_1/N_1$ act by conjugation on $N_0/N_1$. This gives us the second map in the following:
 	$$G_1 \onto G_1/N_1 \to \Aut(N_0/N_1).$$
 	The image in $\Aut(N_0/N_1)$ contains $\Inn(N_0/N_1)\simeq N_0/N_1$.
 	We take $G_0$ to be the preimage of $\Inn(N_0/N_1)$ in $G_1$.
 	Then $G_0$ surjects onto $N_0/N_1$, and hence onto $A_n$, as required.
 \end{proof}

 Lemma \ref{lem:technical_lemma} implies
 the following  proposition, saying that
 the automorphism group of any RAAG $\A_\Gamma$ involves all finite groups, except maybe for $\A_\Gamma\simeq\bbZ^n$.
 Note that the argument says nothing about the group of outer automorphisms.

 \begin{prop}\label{prop:raag fgi}
 	If $\A_\Gamma$ is a non-abelian RAAG, then $\Aut(\A_\Gamma)$  involves all finite groups.
 \end{prop}

 \begin{proof}
 	Considering the automorphism group of a RAAG $\A_\Gamma$, we have the
 	short exact sequence
 	$$1\ra \A_\Gamma/Z(\A_\Gamma) \ra \Aut(\A_\Gamma) \ra \Out(\A_\Gamma)\ra 1$$
 	where $Z(\A_\Gamma)$ is the centre of $\A_\Gamma$, and
 	$\A_\Gamma/Z(\A_\Gamma) \cong
 	\Inn(\A_\Gamma)$.
 	Since $\A_{\Gamma}$ is
 	non-abelian, it maps onto $\bbF_2$, and the kernel of
 	this map necessarily contains $Z(\A_\Gamma)$.
 	Hence $\A_\Gamma/Z(\A_\Gamma)$ involves all finite groups, and so does $\Aut(\A_\Gamma)$ by the second part of Lemma \ref{lem:technical_lemma}.
 \end{proof}

On the other hand, it is well known that because of the congruence subgroup property, for $n\geq 3$,
not all  finite  subgroups are involved in $\SL_n(\bbZ)$  (see for instance \cite{LoRe_surface}, using \cite[Prop 16.4.10 p.348]{LubSeg_subgroup} or \cite[Thm 5.7A]{DM_permutation_groups}).

\begin{thm}[\cite{LoRe_surface,DM_permutation_groups}]\label{thm_SLn}
 	For $n\geq 3$, not all finite groups are involved in $\SL_n(\bbZ)$.
\end{thm}

Proposition \ref{prop:raag fgi} says nothing about the outer automorphism group of a RAAG.
For the group $\Out(\bbF_n)$, Grunewald and Lubotzky have exhibited a nice sequence of representations.

 \begin{thm}[{\cite[Th 9.2]{GruLub_linear}}]\label{thm:Grunewald-Lubotzky}
 	Let $\bbF_n$ be a free group of rank  $n\geq 2$.
 	For every $h\geq 1$ there is a finite index subgroup $G_h$ of  $\Out(\bbF_n)$ and a representation
 	$$\rho_h : G_h \ra \SL_{(n-1)h}(\bbQ)$$
 	whose image is commensurable with $\SL_{(n-1)h}(\bbZ)$.
 \end{thm}

In contrast with $\SL_n(\bbZ)$, this implies:

\begin{cor}[Grunewald--Lubotzky]\label{cor_involved_OutFn} \
 	\begin{itemize}
 		\item $\Out(\bbF_2)$ and $\Out(\bbF_3)$ are large.
 		\item  For any $n\geq 2$, all  finite groups are involved in $\Out(\bbF_n)$.
 	\end{itemize}
\end{cor}

 \begin{proof}
 	The first assertion is clear for $\Out(\bbF_2)\simeq GL_2(\bbZ)$.
 	Theorem \ref{thm:Grunewald-Lubotzky} applied to $\Out(\bbF_3)$ for $h=1$ shows that $\Out(\bbF_3)$ is large.
 	
 	The second assertion also follows from Theorem \ref{thm:Grunewald-Lubotzky}. Indeed, by strong approximation,
 	any finite index subgroup $H<\SL_{(n-1)h}(\bbZ)$ maps onto $\SL_{(n-1)h}(\bbZ/p\bbZ)$ for some prime number $p$  \cite{MVW_congruence}.
 	Since this  finite group contains the  alternating group $A_{(n-1)h}$, $H$ has finite index subgroup
 	 that surjects onto $A_{(n-1)h}$. Applying this to $\Gamma=\rho_h(G_h)$ proves the corollary,  via Remark \ref{rem:afgi An}.\end{proof}

 We will use strategies similar to the construction used in the proof of Theorem \ref{thm:Grunewald-Lubotzky}
 to prove the largeness of the group of outer automorphism group of some RAAGs, see Lemma \ref{lem:three partial conjugations}.

 \subsubsection{SQ-universality}

 \begin{dfn}
 	A group $G$ is \emph{SQ-universal} if any countable group embeds in some quotient of $G$.
 \end{dfn}

 Clearly, any group that has an SQ-universal quotient is itself SQ-universal,  thus satisfying \ref{vast:surjection}.
 A result of P.{} Neumann  shows that if $H$ is a finite index subgroup of $G$, then $G$ is SQ-universal if and only if $H$ is as well
 \cite{Neumann_SQ}, giving \ref{vast:finite index}.
 The following lemma confirms that \ref{vast:SES} holds, and thus SQ-universality is a vastness property.

 \begin{lem}\label{lem:SQ SES}
 	Consider a short exact sequence of groups $1\ra N\ra G \ra Q \ra 1$.
 	
 	If $G$ is SQ-universal, then so is either $N$ or $Q$.
 \end{lem}

 \begin{proof}
 	Suppose neither $N$ nor $Q$ is SQ-universal.
 	Let $A$ (resp. $B$) be a  countable group that does not embed in a quotient of $Q$ (resp. $N$).
 	By Hall \cite{Hall_embedding},
 	$A*B$ can be embedded in a countable simple group $C$.
 	If $G$ is SQ-universal there is a quotient $p:G\onto H$ such that $C$ embeds in $H$. Then $p(N)$ is a normal subgroup of $H$,
 	so $p(N)\cap C$ is normal in $C$. If $p(N)\supset C$, then $p(N)$ is a quotient of $N$ containing $B$, a contradiction.
 	Thus $p(N)\cap C=1$, and $C$ embeds in $H/p(N)$. Thus  $H/p(N)$ is a quotient of $Q$ containing $A$, a contradiction.
 \end{proof}

 \begin{cor}
 	SQ-universality is a vastness property.\qed
 \end{cor}

 The fact that any countable group embeds in a two generator group says that $\bbF_2$ is SQ-universal, and hence so is any large group.
 In fact, every hyperbolic group, and more generally, any group with a proper hyperbolically embedded subgroup is SQ-universal \cite{Olshanskii_SQ,DGO_HE}.
 Since for $n\geq 2$, $\Out(\bbF_n)$ contains a proper hyperbolically embedded subgroup \cite{BF_hyperbolic,DGO_HE}, we get:

 \begin{thm}[\cite{DGO_HE}] \label{thm_SQ_OutFn}
 	For all $n\geq 2$, $\Out(\bbF_n)$ is SQ-universal.
 \end{thm}

 On the other hand, for $n\geq 3$, every normal subgroup of $\SL_n(\bbZ)$ is either central or has finite index, by the Margulis normal subgroup theorem, implying:
 \begin{thm}[Margulis]\label{thm_SQ_GLn}
 	For  $n\geq 3$, $\SL_n(\bbZ)$ is not SQ-universal.
 \end{thm}

 \subsubsection{Quasimorphisms}\label{prelim:vastness-quasim}

 \begin{dfn}
 	A \emph{quasimorphism} is a map $\phi:G\ra \bbR$ such that there exists $C\geq 0$ for which $\left|\phi(gh)-\phi(g)-\phi(h)\right|\leq C$ for all $g,h\in G$.
 	
 	The quasimorphism is \emph{homogeneous} if for all $g\in G$ and all $k\in \bbZ$, $\phi(g^k)=k\phi(g)$.
 \end{dfn}

 The set of all homogeneous quasimorphisms is an $\bbR$--vector space which we denote by $\QH(G)$.

 \begin{dfn}
 	Say that $G$ \emph{virtually  has many quasimorphisms} if $G$ has a finite index subgroup $H$
 	such that $\QH(H)$ is infinite dimensional.
 \end{dfn}

 Assume that $H$ is finitely generated.
 Then the space $\Hom(H;\bbR)$ of all true homomorphisms $H\ra \bbR$ is a finite dimensional subspace of $\QH(H)$,
 and the quotient $\calq(H)=\QH(H)/\Hom(H;\bbR)$
 is isomorphic to the kernel $\EH^2_b(H;\bbR)$ of the comparison map $H^2_b(H;\bbR)\ra H^2(H;\bbR)$
 from bounded cohomology to usual cohomology.
 Thus, a finitely generated group $G$  virtually has many quasimorphisms if and only if
 $\EH^2_b(H;\bbR)$ is infinite dimensional for some finite index subgroup $H$.

 \begin{prop}
 	Virtually having many quasimorphisms is a vastness property.
 \end{prop}

\begin{proof}
 If $p:G\onto G'$ is an epimorphism, and if $H'<G'$ has finite index, consider the finite index subgroup $H=p\m(H')<G$.
 	Then composition by $p$ yields an injective map $\QH(H')\to \QH(H)$.
 	\ref{vast:surjection} follows.
 	
 	If $H<G$ has finite index $d$, the restriction map $\QH(G)\ra \QH(H)$ is injective.
 	Indeed, assume that the restriction of $\phi \in \QH(G)$ to $H$ is trivial. Then, for $g\in G$, since $g^d\in H$, it follows that $0= \varphi(g^d) = d\varphi(g)$.
 	Thus $\varphi$ is trivial on $G$.
 	This implies that if $\QH(G)$ is infinite dimensional, then so is $\QH(H)$.
 	Using this, it is not hard to verify \ref{vast:finite index}.

 	Let us prove assertion \ref{vast:SES}  (c.{}f.{} Gromov's mapping theorem \cite{Gromov_volume}).
 	Consider a short exact sequence $$1\ra N\ra G \ra Q \ra 1$$ with $N$ not virtually having many quasimorphisms,
 	and consider $H<G$ a finite index subgroup.
 	By restriction, we get a short exact sequence
 	$$1\ra N'\ra H \ra Q' \ra 1$$ where $N'=H\cap N$ and $Q'$ is the image of $H$ in $Q$.
 	Assuming that $\QH(H)$ is infinite dimensional, we will prove that $\QH(Q')$ is infinite dimensional.
 	Since $N$ does not virtually have many quasimorphisms, $\QH(N')$ is finite dimensional,
 	so the kernel $\QH_0(H)\subset \QH(H)$ of the restriction map to $N'$ is infinite dimensional.
 	
 	Consider a quasimorphism $\phi\in \QH_0(H)$, and let $C$ be such that $\abs{\phi(gh)-\phi(g)-\phi(h)}<C$ for all $g,h\in G$.
 	Since $\phi$ vanishes on $N'$,
 	the values of $\phi$ on a coset $gN'$ differ by at most $C$ from $\phi(g)$.
 	In particular,  $\phi(gN')$ is bounded,
 	and we define $\hat\psi:Q'\to \bbR$ by
 	$$\hat\psi(gN'):=\inf\{\phi(gk) \mid k \in N' \}.$$
 	It follows that $\abs{\hat\psi(ghN') - \hat\psi(gN') - \hat\psi(hN')}<4C$.
 	We can then make $\hat\psi$ homogeneous in the usual way, by taking
 	$$\psi(\overline{g}) := \lim\limits_{n\to \infty} \frac{\hat{\psi}(\overline{g}^n)}{n},\ \ \textrm{for $\overline{g}\in Q'$}.$$	
 	Then $\psi$ is the unique homogeneous quasimorphism on $Q'$ that is a bounded distance from $\hat{\psi}$.
 	The map $\phi\mapsto \psi$ from $\QH_0(H)\to QH(Q')$ is  injective: if $\psi=0$, then $\hat\psi$ is bounded, so $\phi$ is bounded too.
 	Since $\phi$ is homogeneous, $\phi=0$.
\end{proof}

 By \cite{Brooks} any non-abelian free group $G$
 has an infinite dimensional space of homogeneous quasimorphisms.
 It follows from the above that any large group virtually has many quasimorphisms.
 This was generalized to hyperbolic groups
 by Epstein--Fujiwara  \cite{EpFu_second}, and to groups acting on a hyperbolic space with a WPD element by Bestvina-Fujiwara \cite{BeFu_quasi}.
 By \cite{BF_hyperbolic}, $\Out(\bbF_n)$ has such an action on a hyperbolic space. Therefore:

 \begin{thm}[\cite{BF_hyperbolic}] \label{thm_H2_OutFn}
 	$\QH(\Out(\bbF_n))$ is infinite dimensional for $n\geq 2$.
 \end{thm}

 On the other hand, for $n\geq 3$, every homogeneous quasimorphism on $\SL_{n}(\bbZ)$ is trivial \cite{Newman_unimodular}.
 This also holds for any finite index subgroup of $\SL_{n}(\bbZ)$ by \cite[Th. 21]{BuMo_continuous}.
 Alternatively, by \cite[Thm. 6.1]{DWM_bounded} every finite index subgroup $H$ of $\SL_n(\bbZ)$ contains a finite index subgroup $H_0$ that is boundedly generated by elementary matrices.
 Triviality of $\QH(H)$ then follows since it embeds into $\QH(H_0)$. Indeed,
 Steinberg relations,
 \begin{equation*}
 [E_{ij}(\lambda),E_{kl}(\mu)] = \begin{cases}
 E_{il}(\lambda\mu) & \textrm{if j=k,} \\
 1 		& \textrm{otherwise,}
 \end{cases}
 \end{equation*}
where $E_{ij}(\lambda)$ represents the elementary matrix which differs from the identity matrix by a $\lambda$ in the $(i,j)$--entry,
show that
elementary matrices in $H_0$ have a power that is a single commutator,
and quasi-morphisms are uniformly bounded on commutators.

 Hence we get:

 \begin{thm}[\cite{Newman_unimodular,BuMo_continuous,DWM_bounded}]\label{thm_H2_GLn}\ \newline
For $n\geq 3$, and any finite index subgroup $H < \SL_n(\bbZ)$, $\QH(H)=\{0\}$.
 \end{thm}

 \subsubsection{Bounded generation}\label{sec:boundedly generated}

\begin{dfn}
	A group $G$ is said to be \emph{boundedly generated} if there are elements $g_1,\ldots, g_k \in G$ such that any $g \in G$ can be written in the form $g=g_1^{\alpha_1}\ldots g_k^{\alpha_k}$ for some $\alpha_1,\ldots, \alpha_k \in \bbZ$.

\end{dfn}

\begin{prop}[{\cite[Prop. 7]{Tavgen90_bddgen}}]
	Not being boundedly generated is a vastness property.
\end{prop}

\begin{proof}
	Since bounded generation is stable under epimorphism, \ref{vast:surjection} is clear.
	The condition \ref{vast:SES} holds since  an extension of two boundedly generated groups is boundedly generated.
	
	It remains to check \ref{vast:finite index}.	
	If $H<G$ is of finite index, and $H$ is boundedly generated by $h_1,\ldots,h_k$, then for a set of coset representatives $g_1,\ldots,g_r$ of $H$ in $G$,
we get that $G$ is boundedly generated by $g_1,\dots,g_r,h_1,\dots h_k$.
	We refer the reader to \cite{Tavgen90_bddgen} for verification of the final part, that if $G$ is boundedly generated then so are its finite index subgroups.
\end{proof}

It is well known, since Carter--Keller \cite{CarterKeller}
that the groups $\SL_{n}(\bbZ)$ are boundedly generated for $n\geq 3$.
On the other hand, since $\Out(\bbF_n)$ has an infinite dimensional space of homogeneous quasimorphisms, it is not  boundedly generated.
Indeed, if a group is boundedly generated by $g_1,\dots, g_k$,
then each homogeneous quasimorphism is determined, up to a bounded error, by its value on each $g_i$.

\begin{thm}[\cite{CarterKeller,BF_hyperbolic}]\hfill
	\begin{enumerate}
		\item For $n\geq 2$, the groups $\Out(\bbF_n)$ are not  boundedly generated.
		\item For $n\geq 3$, the groups $\SL_n(\bbZ)$ are boundedly generated.
	\end{enumerate}
\end{thm}

\subsection{McCool groups}\label{sec:prelim McCool}

Let $\bbF$ be a finitely generated free group,
and $\calc=\{[c_1],\dots,[c_s]\}$ a finite set of conjugacy classes of elements of $\bbF$.
One defines the \emph{McCool group} $\Mc(\bbF,\calc)\subset \Out(\bbF)$ as
the set of all outer automorphisms in $\Out(\bbF)$ that preserve  each conjugacy class $[c_i]$.
In fact, if we are given a family of subgroups $(K_i)_{i\in I}$ of $\bbF$,  where $I$ may be infinite,
one can consider the set of all outer automorphisms $\Phi$
such that for all $i$, $\Phi$ has a representative in $\Aut(\bbF)$ fixing $K_i$. This seemingly more general class of subgroups of $\Out(\bbF)$ actually coincides with the class of McCool groups of $\bbF$ \cite{GL_McCool}.

A graph of groups decomposition of $\bbF$ is \emph{relative to $\calc$} if for each $[c]\in \calc$, $c$ is contained in a conjugate
of a vertex group; equivalently, $c$ is elliptic in the Bass-Serre tree of the splitting.
If $H<\bbF$ is a free factor of $\bbF$, and if $c\in \bbF$ has a conjugate in $H$, then
$H\cap [c]$ is an $H$--conjugacy class. We denote by $\calc_{|H}$ the set of conjugacy classes obtained in this way from $\calc$
(note that $\calc_{|H}$ might be empty even if $\calc\neq \es$).

A group $G$ is said to be \emph{freely indecomposable relative to $\calc$} if $G$ does not admit a free product decomposition $G=G_1*G_2$ relative to $\calc$.
A \emph{Grushko decomposition} of $\bbF$ relative to $\calc$ is a free product decomposition
$\bbF = H_1 * \dots * H_k * \bbF_r$ relative to $\calc$
such that each $H_i$ is freely indecomposable relative to $\calc_{|H_i}$.
In the Grushko decomposition, the free product with $\bbF_r$ should be thought of as
$r$ HNN extensions over the trivial group.
We refer to \cite{GL_McCool} or to \cite{GL3} for further details.

The following is a special case of a result of the first author and Levitt for our situation \cite[Thm 4.6]{GL6}.

\begin{thm}[\cite{GL6}]\label{thm:mccool ses}
	Suppose $\bbF$ is freely indecomposable relative to $\calc$.
	Then  there is a finite index subgroup $M^0$ in $\Mc(\bbF;\calc)$ which fits in a short exact sequence
		$$1 \to \calt \to M^0 \to \prod\limits_{j=1}^{s} \MCG(S_j) \to 1$$
	where $\calt$ is a finitely generated abelian group, and
	$\MCG(S_j)$ is the group of isotopy classes of homeomorphisms of a (maybe non-orientable) hyperbolic surface $S_j$,
	 which map each boundary component to itself in an orientation-preserving way.
\end{thm}

\begin{rem}\label{rem_trivial_JSJ}
  The short exact sequence comes from the analysis of the cyclic JSJ decomposition of $\bbF$ relative to $\calc$,
and $\calt$ is a group of twists on the edges.
In particular, if the JSJ decomposition is trivial then $\calt$ is trivial, and $M$ is either finite
(if the JSJ decomposition consists of a single rigid vertex in which case $s=0$),
or $M$ is isomorphic to some $MCG(S)$ for some surface as above (if the JSJ decomposition consists of a
single quadratically hanging vertex, in which case $s=1$). In the latter case, one can take $M=M^0$ but we won't need this fact.
\end{rem}

\subsection{Automorphisms of splittings}
\label{sec:prelim usable results}

We give here two technical results concerning the structure of certain (outer and non-outer) automorphism groups of free products of HNN extensions.

All the $G$--trees $T$ we will consider will be \emph{minimal}: there is no invariant subtree.
A $G$--tree $T$ is non-trivial when $T$ is not a point. In all cases we consider, $T$ will be \emph{irreducible} meaning that there exist two hyperbolic elements whose axes have compact intersection.

\begin{dfn}
  The $G$--tree $T$ is invariant under an automorphism $\phi\in \Aut(G)$
if there exists a map $J_{\phi}:T\ra T$ which is $\phi$--equivariant: $J_\phi(gx)=\phi(g)J_\phi(x)$.

The map $J_{\phi}$ is unique when $T$ is irreducible.
\end{dfn}

Note that $T$ is always invariant under every inner automorphism $\ad_g$ as one can take $J_{\ad_g}(x)=gx$.
Thus, one also says that $T$ is invariant under an outer automorphism $\Phi\in \Out(G)$ if $T$ is invariant under any representative of $\Phi$ in $\Aut(G)$.

Consider a $G$--tree $T$ which is invariant under a group $\cala\subset \Aut(G)$ containing $\Inn(G)$.
Then one gets an action of $\cala$ on $T$ where $\phi\in \cala$ acts through $J_\phi$.
In particular, this action of $\cala$ extends the action of $G$ in the following sense:
for each $g\in G$, the action of $g$ coincides with the action of the inner automorphism $\ad_g\in \cala$.

We remind the reader that the \emph{deformation space} of a $G$--tree $T$ consists of all $G$--trees $T'$ which have the same elliptic subgroups as $T$.
For instance, if $T$ is the $\bbF$-tree associated to a Grushko decomposition of $\bbF$ relative to $\calc$ as above,
then in general, $T$ is not invariant under $\Mc(\bbF,\calc)$, but its deformation space is.

However, in the particular case where the relative Grushko decomposition is of the form $\bbF=A*B$ or $\bbF=A*_{\grp1}$, then
its Bass-Serre tree $T$ is in fact invariant under $\Mc(\bbF,\calc)$.
The reason is that $T$ is the \emph{unique} reduced tree in its deformation space by \cite{Lev_rigid}.
Here, a $G$--tree $T$ is \emph{reduced} if whenever the stabiliser of an edge $e$ in $T$ is equal to the vertex stabiliser of one of its end-points, then both end-points of $e$ lie in the same $G$--orbit.
We refer the reader to \cite{Lev_rigid} for details.

The following results describe consequences of this invariance for the structure of certain groups of
(outer and non-outer) automorphisms of free products or HNN extensions,
showing that they themselves split as amalgamated products or HNN extensions.
Note that these results would not extend to graphs of groups with more factors.

\begin{lem}\label{lem_usable_amalg}
  Let $H=A*B$ with $A,B$ non-trivial groups.
Let $\calo$  be a group of outer automorphisms of $H$ that preserve the conjugacy classes of $A$ and of $B$. Let $\cala$ be the preimage of $\calo$ in $\Aut(H)$.
Then
\begin{enumerate}
	\item $\calo\hookrightarrow \Aut(A)\times \Aut(B)$. The image of $\Phi \in \calo$ is obtained as follows: there is a unique respresentative $\Tilde \Phi\in \Aut(H)$
	preserving $A$ and $B$ (not just their conjugacy classes) simultaneously.
	Then the image of $\Phi$ is given by $(\tilde \Phi_{|A} , \tilde \Phi_{|B})$, the restrictions of $\Tilde \Phi$ to $A$ and $B$.
	
	\item $\cala$ splits as the amalgam $$\cala\simeq \Big[\calo \semidirect A\Big]*_\calo \Big[ \calo\semidirect B\Big]$$
	where the action of $\calo$ on $A$ and $B$ is given by the embedding in $\Aut(A)\times \Aut(B)$.
\end{enumerate}
\end{lem}

\begin{proof}
  Let $T$ be the Bass--Serre tree of the decomposition $H=A*B$.
Since the conjugacy classes of $A$ and $B$ are $\calo$--invariant, $\calo$ preserves the deformation space of this splitting.
By \cite{Lev_rigid}, $T$ is the unique reduced tree in its deformation space, so $T$ is $\cala$--invariant.
As above, the isometries $J_\phi$ for $\phi\in \cala$ yield an action of $\cala$ on $T$ which extends the action of $H$.

Let $a,b\in T$ be the vertices fixed by $A$ and $B$, and $e$ the edge joining them.
For $\Phi \in \calo$, let $\Tilde \Phi \in \cala$ be the unique preimage of $\Phi$ in $\Aut(H)$
such that $J_{\Tilde \Phi}$ fixes $e$.
Such $\Tilde \Phi$ exists because there is only one $H$--orbit of edges and because $\Phi$ does not exchange the conjugacy classes of $A$ and $B$,
while uniqueness is a consequence of the fact that the $H$--stabilizer of $e$ is trivial.
Note that $\Tilde \Phi$ can be alternatively described as in the statement of the lemma.
The mapping $\Phi \mapsto (\Tilde \Phi_{|A},\Tilde \Phi_{|B})\in \Aut(A)\times \Aut(B)$ is a homomorphism, and it is one-to-one
because if $\Phi$ is in the kernel, then $\Tilde\Phi$ fixes $A$ and $B$ and is therefore the identity.
This proves the first part.

To prove the second part, we use the action of $\cala$ on $T$ described above.
Since $\cala$ contains $\Inn(H)$ and contains no element exchanging $a$ and $b$,
the quotient graph $T/\cala$ is a single edge, so the action of $\cala$ on $T$ yields a splitting of $\cala$ as an amalgamated product.
The stabilizer $\cala(e)$ of the edge $e$ is exactly the set of lifts $\Tilde \Phi$, as defined above, for $\Phi \in \calo$.
The stabilizer $\cala(a)$ of $a$ is the set of automorphisms $\phi\in \cala$ preserving $A$.
It it is generated by $\cala(e)$ and $\Tilde A:=\{\ad_x \mid x\in A\}\subset \Aut(H)$, with $\Tilde A\simeq A$ since $H$ has trivial center.
The group $\Tilde A$ is normal in $\cala(a)$, and under the identification of $\cala(e)$ with $\calo$ and of $\Tilde A$ with $A$,
one has $\cala(a)\simeq \calo\semidirect A$.
 A similar statement holds for the stabilizer of $b$, and the lemma follows.
\end{proof}

The following lemma gives a similar result for HNN extensions over the trivial group.
As we will see in the proof,  the subgroup $\calo^0$  is actually the subgroup (of index at most 2) of $\calo$
that does not switch the orientation of the edge of the quotient graph.

\begin{lem}\label{lem_usable_HNN}
  Let $H=A*_{\{1\}}$ be an HNN extension over the trivial group, and $\calo$ a group of outer automorphisms of $H$, preserving the conjugacy class of $A$.
Denote by $t$ the stable letter of the HNN extension.

Let $\calo^0\subset \calo$ be the set of all $\Phi\in \calo$ sending the conjugacy class of the stable letter $t$
to the conjugacy class of an element of $tA$.
Let $\cala^0$ be the preimage of $\calo^0$ in $\Aut(H)$.
Then
\begin{enumerate}
	\item $\calo^0$ is a subgroup of index at most $2$ in $\calo$, and
 $\calo^0\hookrightarrow \Aut(A)\semidirect A$.
The image of $\Phi\in \calo^0$ is obtained as follows:
there is a unique representative $\Tilde \Phi\in \Aut(H)$
preserving $A$ and sending $t$ to $ta$ for some $a\in A$.
The image of $\Phi$ is $(\Tilde \Phi_{|A},a)\in \Aut(A)\semidirect A$.
	
	\item
	Suppose that $A$ is abelian, and let $\cala'$ be the image of $\calo^0$ in $\Out(A)=\Aut(A)$. Then $\cala^0$ embeds in the HNN extension
			$$\cala^0 \hookrightarrow \Big[ \cala' \ltimes (A\times A)\Big] *_{\cala'\ltimes A}$$
	where the embeddings $j_1,j_2$ defining the HNN extension are
	the natural embeddings $$j_1:\cala'\semidirect A\hookrightarrow\cala'\semidirect \Big[ \{1\}\times A\Big]\ \textrm{ and }\ j_2:\cala'\semidirect A\hookrightarrow\cala'\semidirect \Big[ A\times \{1\} \Big].$$
	
	Whenever $\cala$ contains all automorphisms that are the identity on $A$ and send $t$ to $atb$, for  all $a,b\in A$, we have equality between $\cala^0$ and the HNN extension.
\end{enumerate}

\end{lem}

\begin{rem}
	We note that the assumption that $A$ is abelian in the second part of Lemma \ref{lem_usable_HNN} could be avoided,
but with a more complicated description of $j_1,j_2$.
\end{rem}

\begin{proof}
  Let $T$ be the Bass-Serre tree of the HNN decomposition $H=A*_{\{1\}}$.
Since the conjugacy class of $A$ is $\calo$--invariant, $\calo$ preserves the deformation space of this splitting.
By \cite{Lev_rigid}, $T$ is the unique reduced tree in its deformation space, so $T$ is $\cala$--invariant.
As above, the isometries $J_\phi$ for $\phi\in \cala$ yield an action of $\cala$ on $T$ which extends the action of $H$.
In particular, $\calo$ acts on the quotient circle $T/H$ and
the subgroup $\calo^0$ is the subgroup of index at most 2 preserving the orientation of the edge.

Let $v\in T$ be the vertex fixed by $A$, and $e$ the edge joining $v$ to $tv$.
For $\Phi \in \calo^0$, let $\Tilde \Phi \in \cala^0$ be the unique representative fixing $e$.
Since $J_{\Tilde \Phi}$ fixes $e$, it fixes its endpoint $tv$, so
$tv=J_{\Tilde \Phi}(tv)=\Tilde\Phi(t) J_{\Tilde \Phi}(v)=\Tilde\Phi(t) v$, so $\Tilde\Phi(t)=ta$ for some $a\in A$.
The mapping $\Phi\mapsto (\Tilde \Phi_{|A},a)\in \Aut(A)\semidirect A$ is a homomorphism, and it is one-to-one
because if $\Phi$ is in the kernel then $\Tilde\Phi$ fixes $A$ and $t$ and is therefore the identity.
This proves the first part.

To prove the second part, as a first step we show that the action of $\cala^0$ on $T$, as described above, gives the splitting of $\cala^0$ as the HNN extension
		$\cala^0= \left[\cala^0(A)\right]*_{\calo^0}$,
	where $\cala^0(A)\subset \Aut(H)$ is the subgroup of automorphisms in $\cala^0$ preserving the group $A$ (not up to conjugacy),
	and
	where the two embeddings $i_1,i_2 : \calo^0 \to \cala_0(A)$ defining the HNN extension are
	$i_1:\Phi\mapsto \Tilde \Phi$ and
	$i_2:\Phi\mapsto \ad_{t\m}\circ\Tilde \Phi \circ \ad_t$.
Since $\cala^0$ contains $\Inn(H)$, and since no element of $\calo^0$ flips the edge in $T/H$, the quotient graph $T/\cala^0$
is a circle, so we get a splitting of $\cala^0$ as an HNN extension.
The stabilizer of the edge $e$ is exactly the set of lifts $\Tilde \Phi$ for $\Phi \in \calo^0$,
and the stabilizer of the vertex $v$ is $\cala^0(A)$.
Since the endpoints of $e$ are $v$ and $tv$,
one can take for $i_1$ and $i_2$ the maps induced by the inclusion and by the conjugation by $\ad_t\m$.

 Now suppose $A$ is abelian.
Given $\phi\in \cala^0(A)$, consider $a_\phi,b_\phi\in A$ the unique elements such that
$\phi(t)=a_\phi t b_\phi$.
This is possible because $J_\phi$ fixes $v$ and sends $tv$ to a neighbour of $v$,
so $\varphi(t)v =J_\phi(tv) = a_\phi tv$ for some $a_\phi \in A$.

This yields a homomorphism $\cala^0(A)\ra \Aut(A) \semidirect (A\times A)$
defined by $\phi\mapsto (\phi_{|A},a_\phi,b_\phi)$
(if $A$ was not abelian, we would rather need to use $\phi\mapsto (\phi_{|A},a_\phi,b_\phi\m)$,
and this would require changing the description of the embeddings $j_1,j_2$).
This homomorphism is clearly injective,
and it is onto whenever $\cala$ contains all automorphisms fixing $A$ and sending $t$ to $atb$, for all $a,b\in A$.

The subgroup $\cala^0(e)$ consists of the automorphisms in $\cala^0(A)$ with $a_\phi=1$,
so the inclusion $\cala^0(e)\subset\cala^0(A)$ corresponds to the embedding into $\Aut(A) \semidirect (\{1\}\times A)$.
Conjugation by $\ad_{t\m}$ yields another embedding of  $\cala_0(e)$ into $\cala_0(A)$:
if $\phi \in \cala_0(e)$ sends $t$ to $tb_\phi$,
then $\ad_{t\m}\circ\phi\circ\ad_{t}$ sends $t$ to $\ad_{t\m}(tb_\phi)=b_\phi t$.
This yields the embedding $j_2$.
\end{proof}

\section{Using equivalence classes} \label{sec:free equivs}

In this section we prove the first alternative of the trichotomy.
It is based on the following simple observation, which we did not find in the literature (compare with \cite[\S 3]{CV09}).

Recall that $\Out^0(\A_\Gamma)$ is the finite index subgroup of $\Out(\A_\Gamma)$ generated by transvections and partial conjugations.

\begin{prop} \label{prop_tricho1}
Let $[v]$ be an equivalence class of $\Gamma$  of size $n\geq 2$.

Then there is a natural epimorphism $\Out^0(\A_\Gamma)\onto \Out^0(\A_{[v]})$
where $\Out^0(\A_{[v]})$ is isomorphic to $\SL_n(\bbZ)$ or $\Out^+(\bbF_n)$ according to whether $[v]$ is abelian or not.
\end{prop}

\begin{rem}\label{rem:add inversions}
	If we consider the finite index subgroup $\Out^1(\A_\Gamma)$ of $\Out(\A_\Gamma)$ generated by transvections, partial conjugations, and inversions, then
	Proposition \ref{prop_tricho1} can be modified to show $\Out^1(\A_\Gamma)$ maps onto $\GL_n(\bbZ)$ or $\Out(\bbF_n)$ instead.
\end{rem}

The proof of Proposition \ref{prop_tricho1} will be given below, after defining and discussing the restriction and factor maps.

\subsection{Restriction and factor maps}\label{sec_res_fact}

Let $\Gamma'$ be an induced subgraph of $\Gamma$.
The inclusion induces a monomorphism $\A_{\Gamma'}\hookrightarrow \A_{\Gamma}$, allowing us to identify
$\A_{\Gamma'}$ with a subgroup of $\A_\Gamma$.

Given any outer automorphism $\Phi \in\Out(\A_\Gamma)$ preserving the conjugacy class of $\A_{\Gamma'}$,
choose a representative $\phi\in \Aut(\A_\Gamma)$ of $\Phi$
preserving $\A_{\Gamma'}$, and look at its restriction  $\phi_{|\A_{\Gamma'}}$.
The choice of $\phi$ is well defined up to the normalizer of $\A_{\Gamma'}$.
Since
 by \cite[Prop 2.2]{CCV_automorphisms}, the normalizer of $\A_{\Gamma'}$  is generated by $\A_{\Gamma'}$ and its centralizer,
it follows that the class of $\phi_{|\A_{\Gamma'}}$ in $\Out(\A_{\Gamma'})$
does not depend on the choice of the representative $\phi$.
Denoting by $\Out(\A_\Gamma;\A_{\Gamma'})$
the group of outer automorphisms that preserve the conjugacy class of $\A_{\Gamma'}$,
we thus get a \emph{restriction map} (as considered in \cite{CV09}, for instance)
$$\Res:\Out(\A_\Gamma;\A_{\Gamma'}) \ra \Out(\A_{\Gamma'}).$$

In some situations, the whole group $\Out^0(\A_\Gamma)$ preserves the conjugacy class of $\A_{\Gamma'}$.
This happens precisely when for every $u \in \Gamma \setminus \Gamma'$ we have both:
\begin{itemize}
	\item  $u \not\geq v$ for any $v \in \Gamma'$,
	\item  there is a connected component $Z$ of $\Gamma \setminus \st(u)$ such that the vertex set of $\Gamma'$ is contained in $Z\cup \st(u) $. 
\end{itemize}
The former ensures that all transvections preserve $\A_{\Gamma'}$, while the latter says all partial conjugations preserve $\A_{\Gamma'}$, up to conjugacy.
In such a situation, $\Out^0(\A_\Gamma)=\Out^0(\A_\Gamma;\A_{\Gamma'})$, so the restriction map is indeed
$$\Res:\Out^0(\A_\Gamma) \ra \Out^0(\A_{\Gamma'}).$$

In addition to the inclusion $\A_{\Gamma'}\hookrightarrow \A_{\Gamma}$,
there is also an epimorphism $\kappa: \A_{\Gamma}\ra \A_{\Gamma'}$ obtained by killing every vertex in $\Gamma\setminus \Gamma'$
($\kappa$ is a retraction of the inclusion).
If the kernel of $\kappa$ is preserved by $\Out^0(\A_\Gamma)$, then
one gets a homomorphism
$$\Fact:\Out^0(\A_\Gamma) \ra \Out^0(\A_{\Gamma'})$$
which we call a \emph{factor map}.
Since partial conjugations always factor through $\kappa$, the factor map is well-defined precisely if the following condition holds:
$$ v\in \Gamma',\ u\leq v \imp u\in \Gamma'.$$
Thus, given any set of vertices $S$ in $\Gamma$, and defining $\Gamma_{\leq S}$ to be the subgraph of $\Gamma$ induced by vertices $u\leq v$ for some $v\in S$,
then $\Fact:\Out^0(\A_\Gamma) \ra \Out^0(\A_{\Gamma_{\leq S}})$ is always well-defined.

We note that $\Fact$ and $\Res$ send transvections to transvections or to the identity, and send
partial conjugations to  partial conjugations or the identity.

\subsection{Mapping to automorphisms of equivalence classes}

We are now ready to prove Proposition \ref{prop_tricho1}.

\begin{proof}[Proof of Proposition \ref{prop_tricho1}.]
	 Suppose  $V=[v]$ has size at least two and
	consider   $\Gamma_{\leq V}$  and the factor map $\Fact:\Out^0(\A_\Gamma)\ra \Out^0(\A_{\Gamma_{\leq V}})$ as defined above.
	
We claim that  the image of $\Out^0(\A_\Gamma)$ in $\Out^0(\A_{\Gamma_{\leq V}})$
	preserves the conjugacy class of $\A_V$.
Indeed, if $\tau\in \Out^0(\A_\Gamma)$ is a transvection with multiplier $u$,
        then $\Fact(\tau)=\id$ whenever $u\not\leq v$;
if $u\leq v$ but $u\notin V$, then $\Fact(\tau)$ is the identity on $\A_V$;
        in the remaining case, $u\in V$ so $\Fact(\tau)$ stabilizes $\A_V$.

   Let $c$ be a partial conjugation of $\A_\Gamma$ with multiplier $w$, and let us prove that the conjugacy class of $V$ is invariant under $c$.
   This holds whenever $V$ is contained in $Z\cup \st(w)$ for some connected component $Z$ of $\Gamma_{\leq V}\setminus \st(w)$.
   Assume the contrary, that
   $x,y\in V$   are separated by $\st(w)$.
    Then $V$ must be a free equivalence class, and
   $\lk(x)=\lk(y)\subset \st(w)$, so $v\leq w$. If $w\in V$, then $\Fact(c)$ preserves $\A_{V}$.
If $w\notin V$ then $w\notin \Gamma_{\leq V}$, so $c$ maps to the identity in $\Out^0(\A_{\Gamma_{\leq V}})$
   and we are done.

Thus one can compose $\Fact$ with the restriction map
 $\Res:\Out^0(\A_{\Gamma'};\A_{V})\ra \Out^0(\A_V)$, and get
 the desired homomorphism $\Res\circ \Fact:\Out^0(\A_\Gamma) \to \Out^0(\A_V)$.
 To see that it is surjective, we just observe that all transvections and partial conjugations of $\A_V$ are in the image.
\end{proof}

\section{Separating intersections of links}\label{sec:SIL}

The goal of this section is the second alternative of our trichotomy.

\begin{prop}\label{prop_tricho2}
  Let $\Gamma$ be a graph which contains a SIL, and in which all equivalence classes are abelian.

Then $\Out(\A_\Gamma)$ is large.
\end{prop}

To prove this result,
we first introduce a particular kind of SIL, which we call a \emph{special SIL}.
We then prove that if $\Gamma$  satisfies the hypotheses of Proposition \ref{prop_tricho2}
then it has a special SIL (Proposition \ref{prop:all abelian, sil->special}).
We complete this section by showing that if $\Gamma$ has a special SIL, then $\Out(\A_\Gamma)$ is large. This
last part does not use the fact that all the equivalence classes are abelian, see Proposition \ref{prop:large}.

\subsection{SILs and special SILs}\label{subsec_SIL}

\begin{dfn}[\cite{GPR_automorphisms}]
  A \emph{separating intersection of links (SIL)} is a triple $\left( x_1 ,x_2 \mids x_3 \right)$ of vertices in $\Gamma$
  which pairwise do not commute, and such that the connected component  of $\Gamma\setminus (\lk(x_1)\cap\lk(x_2))$ containing $x_3$ avoids $x_1$ and $x_2$.
\end{dfn}

Note that $(x_1,x_2\mids x_3)$ is a SIL if and only if $(x_2,x_1\mids x_3)$ is a SIL. Other permutations of the triple do not necessarily preserve the SIL.

Consider a SIL  $S=(x_1,x_2 \mids x_3)$ and $Z$
the connected component of $\Gamma\setminus (\lk(x_1)\cap\lk(x_2))$ containing $x_3$.
We recall from the introduction that the two partial conjugations $C_Z^{x_1},C_Z^{x_2}$
generate a free group in $\Out(\A_\Gamma)$. We call them the \emph{SIL automorphisms} of $S$.

 We proceed to give, in Lemma \ref{lem:shared component}, an alternative definition of a SIL that will sometimes be more convenient.
It says that when you have a SIL $(x_1,x_2 \mids x_3)$,
there is a subset of $\Gamma$ that appears as a connected component of
each of the three sets $\Gamma\setminus \st(x_1)$, $\Gamma\setminus \st(x_2)$, and $\Gamma \setminus (\lk(x_1)\cap\lk(x_2))$.
 It follows from \cite[Lemma 4.5]{GPR_automorphisms}, and we also refer the reader to \cite[Section 2.3]{DayWade_BNS}, however we include a short proof for completeness.

 We first establish some notation.
 Given $u,v\in \Gamma$ that don't commute,
 denote by  $\Gamma^v_u$ the connected component $\Gamma\setminus \st(v)$ containing $u$.
 With this notation, the partial conjugation $C^v_U$ that acts non-trivially on $u$ has support $U=\Gamma^v_u$.

 We define the \emph{external boundary} $\partial Z$ of a subset $Z\subset \Gamma$ to be the set of vertices in $\Gamma\setminus Z$ that have a neighbour in $Z$.

\begin{rem}\label{rem:component and external boundary}
 The set $\Gamma^v_u$ is characterised as being the connected subgraph of $\Gamma$ that contains $u$, does not intersect $\lk(v)$, and whose external boundary $\partial \Gamma^v_u$ is contained in $\lk(v)$.
 \end{rem}

\begin{lem}\label{lem:shared component}
	A triple $(x,y\mids z)$ is a SIL if and only if  $x,y,z$ don't pairwise commute and $\Gamma^x_z = \Gamma^y_z$.
\end{lem}

\begin{proof}
  If $(x,y\mids z)$ is a SIL, let $Z$ be the connected component of $\Gamma\setminus (\lk(x)\cap \lk(y))$ containing $z$.
  We claim that $\Gamma^x_z=Z$.
  By definition, for every $v\in Z$, all of its neighbours are in $Z \cup (\lk(x)\cap\lk(y))$.
  In particular, its external boundary $\partial Z$ is contained in $\lk(x)$.
  Since $(x,y\mids z)$ is a SIL, $Z$ does not intersect $\st(x)$.
  It follows from Remark \ref{rem:component and external boundary} that $Z = \Gamma^x_z$.
  Similarly $Z = \Gamma^y_z$.
  This proves the first implication.

Conversely, assume that $\Gamma^x_z = \Gamma^y_z$,
and denote this set by $Z$. Then  $\partial Z \subseteq \lk(x)$, and $\partial Z \subseteq \lk(y)$, hence $\partial Z \subseteq \lk(x)\cap\lk(y)$.
Since also $Z$ does not intersect $\lk(x)\cap \lk(y)$, this implies that $Z$ is a connected component of $\Gamma \setminus (\lk(x)\cap \lk(y))$.
Since $x,y\notin Z$,  $(x,y\mids z)$ is a SIL.
\end{proof}

Let $\Gamma_S\subset \Gamma$ be the subgraph induced by the equivalence classes $[x_1]\cup[x_2]\cup [x_3]$.
Assuming that each equivalence class $[x_i]$ is abelian (maybe cyclic), then $\A_{\Gamma_S}$ is the free product
of three abelian groups.
Note that, in this case,
the definition of a SIL then implies that the three equivalence classes are distinct.

Let $\Gamma_{\leq S}\subset\Gamma$ be the subgraph induced by
$\{u\in \Gamma\mid  u\leq x_i \text{ for some }i\leq 3\}$
and consider the corresponding factor map, $\Fact:\Out^0(\A_\Gamma)\ra \Out^0(\A_{\Gamma_{\leq S}})$,  as defined in Section \ref{sec_res_fact}.
Next consider the restriction map $\Res:\Out^0(\A_{\Gamma_{\leq S}};\A_{\Gamma_S})\ra \Out^0(\A_{\Gamma_S})$.
In order to be able to compose these maps, we need that the image of $\Fact$ preserves the conjugacy class of $\A_{\Gamma_S}$.

\begin{dfn}\label{dfn:special_SIL}
	A SIL $S=(x_1,x_2\mids x_3)$ is a \emph{special SIL} if each equivalence class $[x_i]$ is abelian and
	the image of $\Fact:\Out^0(\A_\Gamma)\ra \Out^0(\A_{\Gamma_{\leq S}})$
	is contained in $\Out(\A_{\Gamma_{\leq S}};\A_{\Gamma_S})$.
	
	In particular, the map
	$\rho=\Res\circ \Fact:\Out^0(\A_\Gamma)\ra \Out^0(\A_{\Gamma_{ S}})$ is well-defined.
\end{dfn}

\begin{rem}\label{rem_SIL}
 If $S=(x_1,x_2\mids x_3)$ is a  SIL where each equivalence class $[x_i]$ is abelian
 and with
$\Gamma_{\leq S}=\Gamma_S$, then it is obviously a special SIL.
\end{rem}

\begin{rem}\label{rem_swap}
It is an immediate consequence of the definition that
if $(x_1,x_2\mids x_3)$ is a special SIL, and if $(x_3,x_2\mids x_1)$ happens to be a SIL, then it is necessarily also special. 
\end{rem}

\begin{lem}	\label{lem:combinatorial defn of special sil}
Consider a SIL  $S=(x_1,x_2\mids x_3)$ where the equivalence classes $[x_i]$ are abelian.

Then $S$ is a special SIL if and only if the following hold:
\begin{enumerate}
	\renewcommand{\theenumi}{\textup{(Sp\arabic{enumi})}}
\item\label{item:special SIL 1} for any $u\in \Gamma_{\leq S}$, if $x_i\leq u \leq x_j$ for some $i\neq j\in\{1,2,3\}$, then $u\in \Gamma_S$
\item\label{item:special SIL 2} if $u\in \Gamma_{\leq S}\setminus \Gamma_S$, there is a connected component $Z$ of $\Gamma\setminus \st(u)$ such that $x_1,x_2,x_3$ lie in $Z \cup \st(u)$;
	equivalently, $[x_1]\cup [x_2]\cup [x_3]\subset Z \cup \st(u)$.
\end{enumerate}
\end{lem}

\begin{proof}
Assume \ref{item:special SIL 1} and \ref{item:special SIL 2} hold.
Let $\tau\in \Out^0(\A_\Gamma)$ be a transvection with multiplier $u$. If $u\notin \Gamma_{\leq S}$, its image in $\Out^0(\A_{\Gamma_{\leq S}})$
under the factor map is the identity. Otherwise, $u\leq x_j$ for some $j$. 
We can assume that some element of $\Gamma_S$ not fixed by $\tau$, so $x_i\leq u$ for some $i$.
Condition \ref{item:special SIL 1} guaranties that $u\in \Gamma_S$ if $i\neq j$. 
This also holds if $i=j$, as this implies $u\in [x_i]\subset \Gamma_S$.
It follows that $\Fact(\tau)$ preserves the conjugacy class of $\Gamma_S$.
Now consider a partial conjugation $C_Y^u$ of $\Gamma$. If $u\notin \Gamma_{\leq S}$, then $\Fact(C_Y^u)=\id$.
If $u\in \Gamma_S$, then $\Fact(C_Y^u)$ clearly preserves the conjugacy class of $\A_{\Gamma_S}$.
Otherwise, \ref{item:special SIL 2} gives us that
$[x_1]\cup [x_2]\cup [x_3]\subset Z \cup \st(u)$ for some connected component $Z$ of $\Gamma\setminus \st(u)$,
so  $C_Y^u$ is inner on $\A_{\Gamma_S}$.

Let us now check the equivalence mentioned in \ref{item:special SIL 2}.
Assume that $Z\cup \st(u)$ contains $\{x_1,x_2,x_3\}$.
Then if $x_i\in \st(u)$, then so is any $x'_i\sim x_i$, so $[x_i]\subset \st(u)$.
If $x_i\in Z$ and $x'_i$ is another vertex in $[x_i]$, then $x'_i$ and $x_i$ are joined by an edge because $[x_i]$ is abelian.
Since $Z$ is a connected component of $\Gamma\setminus \st(u)$, and $x_i,x'_i\notin \st(u)$, it follows that $x'_i\in Z$.

Let us now prove the converse implication (which won't be used in the sequel).
Assume that \ref{item:special SIL 2} does not hold, so that there exists $u\in \Gamma_{\leq S}\setminus\Gamma_S$ such that there are $i,j\in \{1,2,3\}$ with $x_i,x_j$
	in two distinct connected components $Y,Y'$, respectively, of $\Gamma\setminus \st(u)$.
Then the partial conjugation $C_Y^u$ sends $x_ix_j$ to $ux_iu\m x_j$,
and one easily checks that it is not conjugate in $\A_{\Gamma_{\leq S}}$ to an element of $\A_{\Gamma_S}$ using, for example, Servatius' solution to the conjugacy problem \cite[p. 38]{Servatius}.

When \ref{item:special SIL 1} does not hold, we have a transvection $L_{x_i}^u$ for some $i\in\{1,2,3\}$ and some $u \in\Gamma_{\leq S}\setminus \Gamma_S$.
Then $L_{x_i}^u$ sends $x_i$ to $x_i u $, which, similarly, is not conjugate in $\A_{\Gamma_{\leq S}}$ to an element of $\A_{\Gamma_S}$.
\end{proof}

\subsection{Finding special SILs}

In this section,  we consider a graph $\Gamma$ in which all equivalence classes are abelian,
 and we prove the existence of a special SIL as soon as there is a SIL.

\begin{prop}\label{prop:all abelian, sil->special}
	Suppose that in a graph $\Gamma$ all equivalence classes are abelian.

If $\Gamma$ has a SIL, then it has a special SIL.
\end{prop}

The general strategy to prove Proposition \ref{prop:all abelian, sil->special} is to choose a SIL $S = (x,y\mids z)$ satisfying some minimality assumption
(implying that $\Gamma_{\leq S}$ is minimal for inclusion among all SILs).
We will verify that such a SIL is special.

The following easy fact will be used several times.

\begin{fact}\label{fact_comm}
Consider $u,x,y\in \Gamma$ such that $x,y$ do not commute.

If $u\leq x$ then $u$ does not commute with $y$. 
\end{fact}

\begin{proof}
  If $u$ commutes with $y$, then $y\in \lk(u)\subset \st(x)$ contradicting that $x$ and $y$ don't commute.
\end{proof}

\begin{lem}\label{lem:(x,y|z) z minimal}
	Let $\Gamma$ be any graph.
 If $(x,y\mids z)$ is a SIL and $z'\leq z$ is such that $z'\notin [x]\cup [y]$,
  then $(x,y\mids z')$ is a SIL.
\end{lem}

\begin{proof}
By Fact \ref{fact_comm}, $z'$ does not commute with $x$ or $y$.
	Let $Z$ be the connected component of $\Gamma\setminus(\lk(x)\cap\lk(y))$ containing $z$.
	If $z' \in Z$, then $(x,y\mids z')$ is a SIL and we are done.
	If on the other hand $z'\notin Z$, then $\lk(x)\cap\lk(y)$ separates $z$ from $z'$, and since $\lk(z')\subset \st(z)$,
	one deduces  $\lk(z')\subset \lk(x)\cap\lk(y)$, see Figure \ref{fig:SIL_to_special-z_min}.
	Thus $\{z'\}$ forms its own connected component in $\Gamma\setminus(\lk(x)\cap\lk(y))$, hence $(x,y\mids z')$ is a SIL.
\end{proof}

 \begin{lem}\label{lem_swap} Let $\Gamma$ be any graph.
	If $(x,y\mids z)$ is a SIL with $x\leq z$, then $(z,y\mids x)$ is also a SIL.
\end{lem}

\begin{proof}
	Since $x\leq z$, and since $x$ and $z$ don't commute by the definition of a SIL, $\lk(x)\subset \lk(z)$.
	We thus have (see Figure \ref{fig:SIL_to_special-z_min2})
	$$\lk(x) = \lk(x)\cap\lk(z)\subseteq \lk(x)\cap\lk(y) \subseteq \lk(z)\cap\lk(y)$$
	where the first  inclusion follows from the fact that $\lk(x)\cap \lk(y)$ separates $x$ from $z$.
	This proves that $(z,y\mids x)$ is a SIL. 
\end{proof}
	
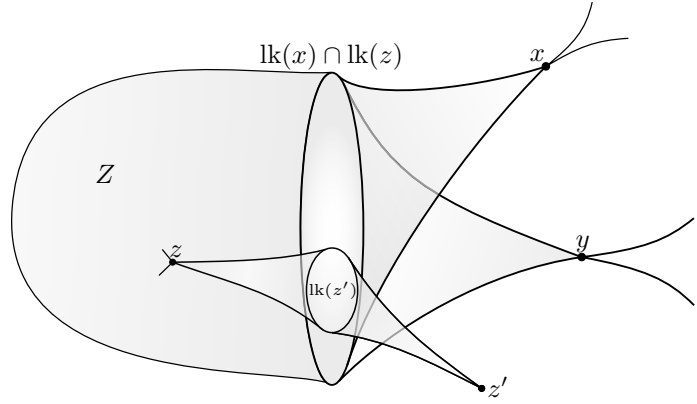
\begin{figure}[ht!]
	\labellist \small
	\pinlabel $x$ at 335 216
	\pinlabel $y$ at 364 95
	\pinlabel $z$ at 105 90
	\pinlabel $z'$ at 310 3
	\pinlabel $Z$ at 60 140
	\pinlabel $\lk(x)\cap\lk(z)$ at 204 215
	\tiny
	\pinlabel $\lk(z')$ at 205 65
	\endlabellist
	
	\centering
	\includegraphics[width=9cm]{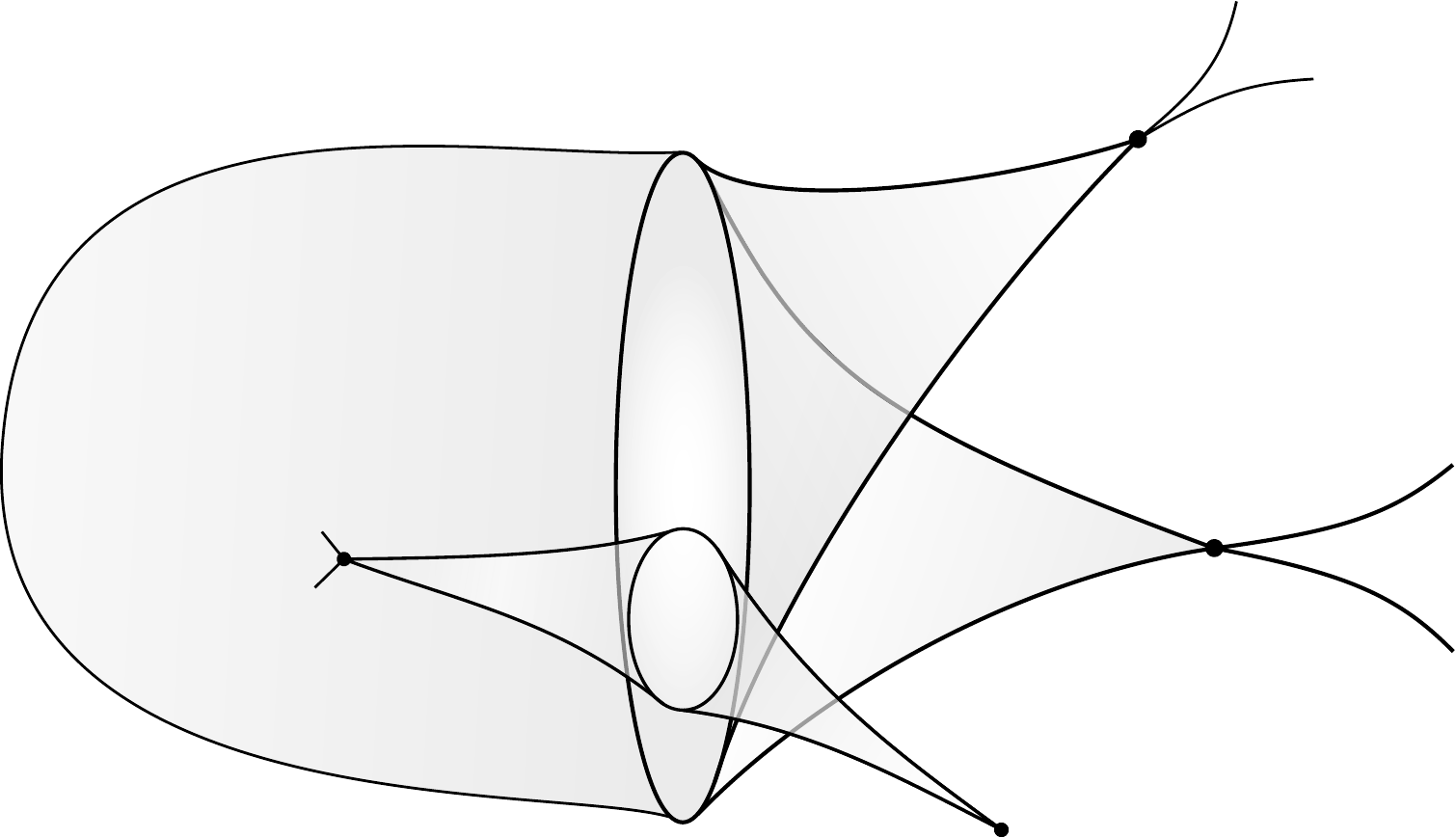}
	\caption{If the set $\lk(x)\cap\lk(y)$ separates $z$ from $z'$, and $\lk(z')\subset \st(z)$, then $\lk(z')\subset \lk(x)\cap\lk(y)$.}
	\label{fig:SIL_to_special-z_min}
\end{figure}

\begin{figure}[ht!]
	\labellist \small
	\pinlabel $x$ at 328 216
	\pinlabel $y$ at 350 95
	\pinlabel $z$ at 52 129
	\pinlabel $\lk(x)=\lk(y)\cap\lk(z)$ at 190 210
	\endlabellist
	
	\centering
	\includegraphics[width=7.5cm]{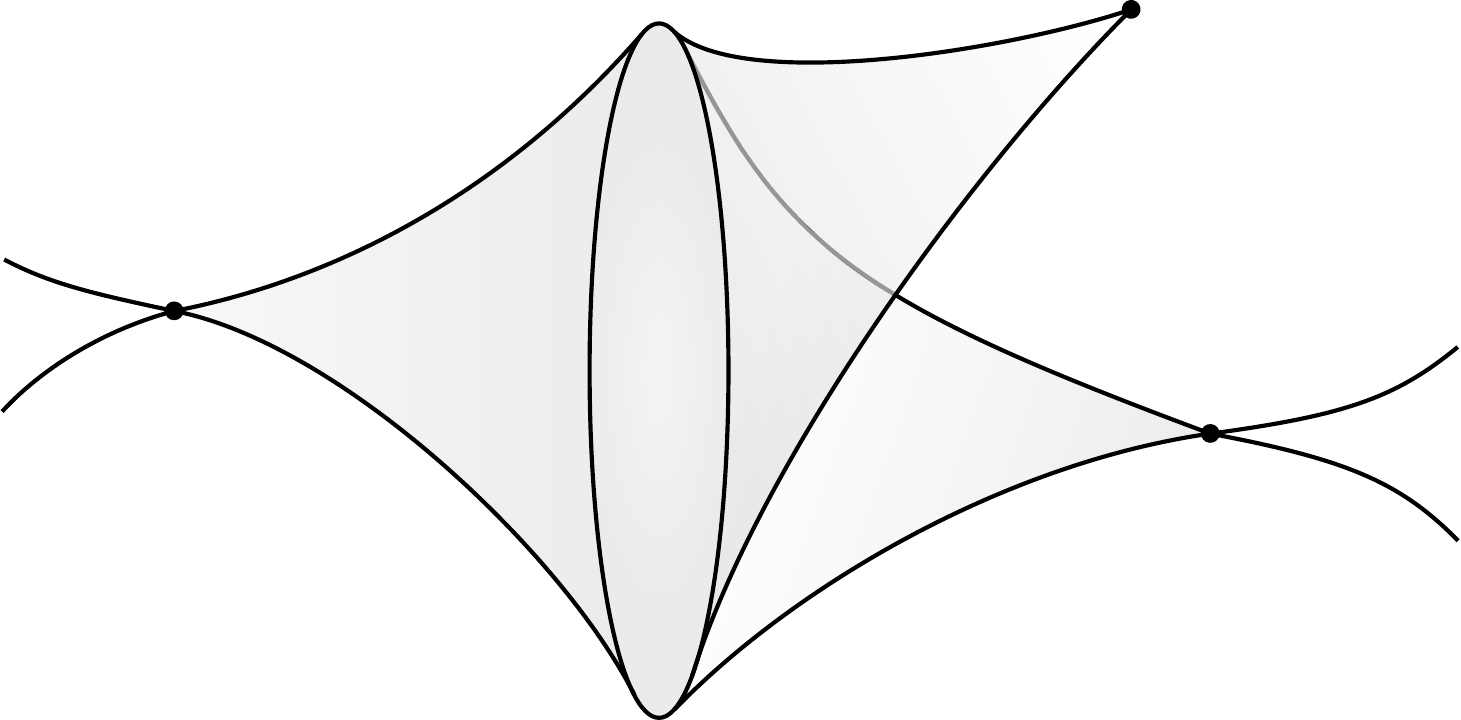}
	\caption{If $x\leq z$ then $\lk(x)\subset\lk(z)$, and since $(x,y\mids z)$ is a SIL, the intersection $\lk(x)\cap\lk(z)=\lk(x)$ must be contained in $\lk(x)\cap \lk(y)$, and hence in $\lk(z)\cap\lk(y)$.}\label{fig:SIL_to_special-z_min2}
\end{figure}
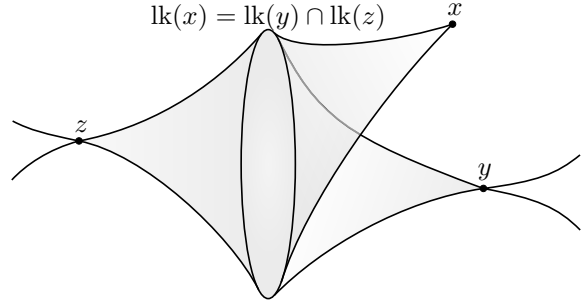

For the following, recall that $\Gamma^u_z$ denotes the connected component of $\Gamma\setminus \st(u)$ that contains $z$.

\begin{lem}\label{lem_sandwich}
	Let $\Gamma$ be any graph. Let $S = (x,y\mids z)$ be a SIL and $u\in \Gamma$ be such that either $z\leq u\leq x$ or $y\leq u \leq x$.

Then $(u,y\mids z)$ is a SIL.
\end{lem}

\begin{proof}
In both cases, as $u\leq x$, Fact \ref{fact_comm} shows that $u$ does not commute with $y$ or $z$.

Assume first that $z\leq u\leq x$.
Since $z\leq x$, the connected component $\Gamma^x_z$ is reduced to $\{z\}$.
Since $z\leq u$, one similarly has $\Gamma^u_z=\{z\}$.
Since $(x,y\mids z)$ is a SIL, one has $\Gamma^y_z=\Gamma^x_z=\{z\}$.
One concludes that $\Gamma^u_z=\Gamma^y_z$, so $(u,y\mid z)$ is a SIL by Lemma \ref{lem:shared component}.

  Assume now that $y\leq u\leq x$.
Let $Z=\Gamma^x_z=\Gamma^y_z$.
Then  the external boundary of $Z$ satisfies $\partial Z\subset \lk(y)\subset \lk(u)$
and $Z$ does not intersect $\lk(u)$ since $\lk(u)\subset \st(x)$.
This shows that $Z = \Gamma^u_z$ by Remark \ref{rem:component and external boundary}.
It follows that $(u,y\mids z)$ is a SIL.
\end{proof}

\begin{lem}\label{lem:u<x in (x,y|z)}
	Let $\Gamma$ be any graph.
	Suppose $S = (x,y\mids z)$ is a SIL and there is a vertex $u$ such that $u\leq x$
	and $u \notin [z]$.
	Then exactly one of the following holds:
	\begin{itemize}
		\item the triple $(u,y\mids z)$ is a SIL, or
		\item there is a connected component $Z$ of $\Gamma\setminus \st(u)$ such that $x,y,z \in Z\cup \st(u)$.
	\end{itemize}
\end{lem}

\begin{proof}

As above, the inequality $u\leq x$ implies, by Fact \ref{fact_comm}, that $u$ does not commute with $y$ or $z$.
	
	Let  $Z=\Gamma^u_z$.
	We claim that either $x,y\in Z$ or $(u,y\mids z)$ is a SIL, but not both.
	Clearly, if $(u,y\mids z)$ is a SIL then $y \notin Z$, so both cannot occur.
	Now assume $(u,y\mids z)$ is not a SIL.
	Then there is a path from $z$ to $\{u,y\}$ avoiding $\lk(u)\cap\lk(y)$.
	Let $p$ be such a path of minimal length, and let $z=p_0,p_1,\ldots,p_k$ be the vertices of $p$.
        This choice of $p$ ensures that for all $i\leq k-2$, $p_i\notin \lk(u)$ and $p_i\notin \lk(y)$.

      We first claim that $p_k\neq u$. 
     Indeed, if $p_k=u$, then $p_{k-1} \in \lk(u)\subseteq \st(x)$,
	so either $(p_0,\ldots, p_{k-1})$ or $(p_0,\ldots, p_{k-1},x)$ defines a path from $z$ to $x$.
	Since $(x,y\mids z)$ is a SIL, this path must go through $\lk(x)\cap\lk(y)$. Since $p_i\notin \lk(y)$ for $i\leq k-2$,
	we get that $p_{k-1}\in \lk(y)$
	and in particular $p_{k-1} \in \lk(u)\cap\lk(y)$, contradicting our choice of $p$.
	
	Thus we assume $p$ joins $z$ to $y$. Since $(x,y\mids z)$ is a SIL, there is some $i\leq k-1$ with $p_{i} \in \lk(x)\cap \lk(y)$.
        As above, since $p_i\notin \lk(y)$ for $i\leq k-2$, we get that $p_{k-1} \in \lk(x)\cap \lk(y)$.
        Since $p$ avoids $\lk(u)\cap\lk(y)$, it follows that $p_{k-1} \notin \lk(u)$, hence $p_{i}\notin \lk(u)$ for all $i\leq k$.
        It follows that the path $p$ connects $z$ to $y$ in the complement of $\lk(u)$, and
        the path $(p_0,\dots,p_{k-1},x)$ connects $z$ to $x$ in the complement of $\lk(u)$ unless $x\in \lk(u)$.
	We conclude that $x,y\in Z\cup \st(u)$.
\end{proof}

\begin{proof}[Proof of Proposition \ref{prop:all abelian, sil->special}]
  Consider a SIL $S=(x_1,x_2\mids x_3)$ 
  with $\{x_1,x_2,x_3\}$ minimal in the following sense: if $(x'_1,x'_2\mids x'_3)$ is another SIL
such that there is a permutation $\sigma$ of $\{1,2,3\}$ with $x'_i\leq x_{\sigma(i)}$ for all $i=1,2,3$, then $x'_i\sim x_{\sigma(i)}$ for all $i$.
We prove that $S$ is special using Lemma \ref{lem:combinatorial defn of special sil}.

Since equivalence classes are abelian, $[x_1]$, $[x_2]$, and $[x_3]$ are distinct.
If $x_1\leq x_3$, then Lemma \ref{lem_swap} allows us to exchange their roles, and since the equivalences classes are distinct, we can assume that $x_1\not\leq x_3$ and $x_2\not\leq x_3$.

To prove \ref{item:special SIL 1} in Lemma \ref{lem:combinatorial defn of special sil},
suppose that $x_i\leq u\leq x_j$, for some $i\neq j\in \{1,2,3\}$, and let us prove that $u\in \Gamma_S$.
Firstly, from above, we may assume $j\neq 3$.
Then, by Lemma \ref{lem_sandwich}, one can replace $x_j$ by $u$ in $S$,
so minimality implies that $u\in [x_j] \subset \Gamma_S$ and \ref{item:special SIL 1} follows.

To prove \ref{item:special SIL 2},
consider $u\in \Gamma_{\leq S}\setminus \Gamma_S$. 
By Lemma \ref{lem:(x,y|z) z minimal}, if $u\leq x_3$ then $(x_1,x_2\mids u)$ is a SIL, contradicting minimality.
If $u\leq x_1$, then by Lemma \ref{lem:u<x in (x,y|z)}, either $(u,x_2\mids x_3)$ is a SIL which contradicts our minimality assumption,
or there exists a component $Z$ of $\Gamma\setminus \st(u)$ such that $x_1,x_2,x_3 \in Z\cup \st(u)$,
and \ref{item:special SIL 2} follows in this case. 
The same argument works if $u\leq x_2$.
\end{proof}

\subsection{Special SILs imply largeness}

The goal of this subsection is the following result. It does not assume that  all
 equivalence classes of $\Gamma$ are abelian,  though we recall that if $(x,y\mids z)$ is a special SIL then, by definition, $[x],[y],[z]$ are each abelian.
\begin{prop}\label{prop:large}
	Assume that the graph $\Gamma$ has a special SIL.
	
	Then $\Out^0(\A_\Gamma)$ is large.
\end{prop}

 Let $(x,y\mids z)$ be a special SIL in $\Gamma$. Recall from Definition \ref{dfn:special_SIL} that $\Gamma_S$ is the induced graph with vertex set  $S=[x]\cup [y]\cup [z]$, and that $\A_{\Gamma_S}$
is the free product of the three abelian groups generated by $[x]$, $[y]$ and $[z]$.
A special SIL comes with a map
$$\rho=\Res\circ\Fact:\Out^0(\A_\Gamma)\ra \Out^0(\A_{\Gamma_S})$$
transiting through $\Out^0(\A_{\Gamma_{\leq S}})$.
Proposition \ref{prop:large} is a direct consequence of the following lemma  showing that the image of $\rho$ in $\Out^0(\A_{\Gamma_S})$ is large.

\begin{lem}\label{lem:special SIL implies large}
	Let $\calo\subset \Out^0(\A_{\Gamma_S})$ be the image of $\rho$.
	
	Then $\calo$ is large.
\end{lem}

The remainder of this section is devoted to proving Lemma \ref{lem:special SIL implies large}.

From now on, we use the following notations: $X=[x]=\{x_1,\dots,x_a\}$,
$Y=[y]=\{y_1,\dots,y_b\}$, and $Z=[z]=\{z_1,\dots,z_c\}$, so that
$S=X\cup Y\cup Z$, and $\A_{\Gamma_S}\simeq \bbZ^a*\bbZ^b*\bbZ^c$.

Thanks to Lemma \ref{lem_swap} and Remark \ref{rem_swap}, we assume without loss of generality that  $x\not \leq z$ and $y\not \leq z$.
Similarly, up to exchanging $x$ and $y$, we can assume that $x\not \leq y$.

\begin{lem}
	Let $\cals$ be the generating set of $\calo$ consisting of non-trivial images under $\rho$ of all transvections and partial conjugations of $\A_\Gamma$.
Then $\cals$ consists of  the following list of outer automorphisms of $\A_{\Gamma_S}$ (up to taking inverses):

	\begin{enumerate}
\renewcommand{\theenumi}{($\cals$\arabic{enumi})}
		\item \label{it_sil}The SIL automorphisms: $C_{Z}^{x_i}:z\mapsto x_izx_i\m$, $C_{Z}^{y_j}:z\mapsto y_jzy_j\m$ for $z\in Z$;
		\item \label{it_int}  The \textit{internal }transvections $R_{x_i}^{x_{i'}},L_{x_i}^{x_{i'}}$; $R_{y_j}^{y_{j'}},L_{y_j}^{y_{j'}}$; $R_{z_k}^{z_{k'}},L_{z_k}^{z_{k'}}$,
			whenever $\card X,\card Y,\card Z\geq 2$ respectively.
		\item \label{it_pc} If $\st(z_k)$ separates $X$ from $Y$, the other partial conjugations $C_{X}^{z_k}:x\mapsto z_kxz_k\m$ for $x\in X$;
		\item \label{it_yx} If $y\leq x$, the transvections $R_{y_j}^{x_i}:y_j\mapsto y_jx_i$, $L_{y_j}^{x_i}:y_j\mapsto x_iy_j$;
		\item \label{it_zx} If $z\leq x$, the transvections $R_{z_k}^{x_i}:z_k\mapsto z_k{x_i}$, $L_{z_k}^{x_i}:z_k\mapsto {x_i}z_k$;
		\item \label{it_zy} If $z\leq y$, the transvections $R_{z_k}^{y_j}:z_k\mapsto z_k{y_j}$, $L_{z_k}^{y_j}:z_k\mapsto {y_j}z_k$.
	\end{enumerate}
\end{lem}

\begin{proof}
	We check the image of all transvections and partial conjugations fit in the list above.
	
Each left transvection $L_v^w$ (respectively $R_v^w$) of $\A_{\Gamma}$ is mapped by $\rho$
	to the corresponding transvection of $\A_{\Gamma_S}$ if $v,w\in \Gamma_S$,
	and to the identity otherwise.
        Since we assume that  $x\not \leq z$, $y\not \leq z$, and $x\not \leq y$,
	any transvection not in the kernel of $\rho$ is mapped to one of the automorphisms of \ref{it_int}, \ref{it_yx}, \ref{it_zx}, or \ref{it_zy}.
	
	Any partial conjugation $C_W^w$ of $\A_\Gamma$ with multiplier $w$ is mapped to the identity under $\rho$ if $w\notin S$.
	If $w\in S$, then its image under $\rho$ is trivial if its support $W$ does not intersect $S$ or if it contains the two equivalence classes from $\{X,Y,Z\}$ which do not contain $w$
	(in the latter case, the image under $\rho$ will be an inner automorphism of $\A_{\Gamma_S}$).
	Alternatively,
	$W\cap S$ contains precisely one of the equivalence classes from $\{X,Y,Z\}$ not containing $w$.
	Note that in $\Out(\A_{\Gamma_S})$, we have the equalities $C_Z^x= (C_Y^x)\m$, and similarly $C_X^y= (C_Z^y)\m$, $C_X^z= (C_Y^z)\m$,
	so that up to taking inverses, there are only 3 kinds of partial conjugations in $\cals$, according to  whether their multiplier lies in $X,Y$ or $Z$.
	Hence partial conjugations are mapped under $\rho$ to either the identity or (up to taking inverse) one of the automorphisms of \ref{it_sil} or \ref{it_pc}.
	
	Finally, the necessity of the automorphisms in \ref{it_sil} follows because $(x,y\mids z)$ is a SIL, meanwhile the necessity of the internal transvections in \ref{it_int} is clear.
\end{proof}

\begin{rem}\label{rem_all_nothing}
  Since all $z_k$ are equivalent, if $\cals$ contains one partial conjugation $C_{X}^{z_k}$ for some $k\leq c$ as in \ref{it_pc},
it contains $C_{X}^{z_k}$ for all $k\leq c$. A similar comment applies to \ref{it_yx}, \ref{it_zx}, \ref{it_zy}.
\end{rem}

We distinguish two cases according to whether $\cals$ contains the partial conjugations $C_{X}^{z_k}$ or not.

\subsubsection{Case 1: when $C_{X}^{z_k} \in \cals$}

In this case $\cals$ contains all the partial conjugations $C_Z^{x_i},C_Z^{y_j}$ and $C_X^{z_k}$ for $i\leq a$, $j\leq b$, $k\leq c$.
Largeness of $\Out(\A_{\Gamma_S})$ then follows from the following lemma,
the proof of which uses a variation of the representations of Looijenga \cite{Looi97} and Grunewald--Lubotzky \cite{GruLub_linear}.

\begin{lem}\label{lem:three partial conjugations}
	Let $a,b,c\geq 1$. Suppose a subgroup $\calo$ of $\Out^0(\bbZ^a * \bbZ^b * \bbZ^c)$ contains the partial conjugations $C_X^y,C_Y^z,C_Z^x$ for some $x\in X, y\in Y, z\in Z$.
	
	Then $\calo$ is large.
\end{lem}

\begin{proof}
	Recall that $\A_{\Gamma_S}= \grp{X}*\grp{Y}*\grp{Z}\simeq \bbZ^a * \bbZ^b * \bbZ^c$.
	Let
	$$\pi : \A_{\Gamma_S} \to \bbZ_2=\grp{g}$$
	 be the homomorphism sending each
         generator in $S$  to the non-trivial element $g$ of the cyclic group $\bbZ_2$.
	Let $\Tilde A=\ker\pi$ and $\Aut_\pi(\A_{\Gamma_S})$ be the finite index subgroup of $\Aut(\A_{\Gamma_S})$
        of automorphisms $\phi$ such that $\pi\circ\phi=\pi$.

	The Grunewald--Lubotzky representation comes from the action of  $\Aut_\pi(\A_{\Gamma_S})$  on the homology group $H_1(\Tilde A;\bbQ)$.
	Let $K$ be the Salvetti complex of $\A_{\Gamma_S}$,
	\ie the wedge of three tori $T_X,T_Y$ and $T_Z$ of dimensions $a,b$ and $c$ respectively.
	Consider the corresponding double cover $\Tilde K_\pi$ of $X$, so that $\bbZ_2$ acts on $\Tilde K_\pi$ by deck transformations (see Figure \ref{fig_tori}).
	The $1$--skeleton of $\Tilde K_\pi$ can be identified with the Cayley graph of $\bbZ_2$ with generating set $\pi(S)=\pi(X\cup Y\cup Z)$.
	For each $w\in X\cup Y \cup Z$, let $e_w$ denote the oriented edge in $\Tilde K_\pi$ joining $1$ to $g$ labelled by $w$.
	The edge joining $g$ to $1$ labelled by $w$ is the translate $ge_w$.	
	Topologically, $\Tilde K_\pi$ is the disjoint union of two-sheeted covers $\Tilde T_X$, $\Tilde T_Y$, and $\Tilde T_Z$ of $T_X,T_Y$ and $T_Z$,
	identified at two distinct points ($1$ and $g$).

	Each automorphism in  $\Aut_\pi(\A_{\Gamma_S})$
	is induced by a homotopy equivalence of $K$, which lifts uniquely to a homotopy equivalence of $\Tilde K_\pi$
	fixing the vertices of $\Tilde K_\pi$, and commuting with the action of $\bbZ_2$ by deck transformation.
	At the level of homology, we get an action of  $\Aut_\pi(\A_{\Gamma_S})$  on $H_1(\Tilde K_\pi;\bbQ)=H_1(\Tilde A;\bbQ)$
	which commutes with the action of $\bbZ_2$, and preserves the lattice $H_1(\Tilde A;\bbZ)$.
	
	The action of $\grp{g}=\bbZ_2$ splits  $H_1(\Tilde A;\bbQ)$  into the direct sum of  eigenspaces
	$H_1(\Tilde A;\bbQ)=V_{+1}\oplus V_{-1}$ corresponding to the eigenvalues $+1$ and $-1$.
	We will show that $V_{-1}\simeq \bbQ^2$
	(this is an extension of a result of G\"aschutz \cite{Gaschutz},
	which applies when $a=b=c=1$).

        \begin{figure}[ht]
          \centering
          \includegraphics{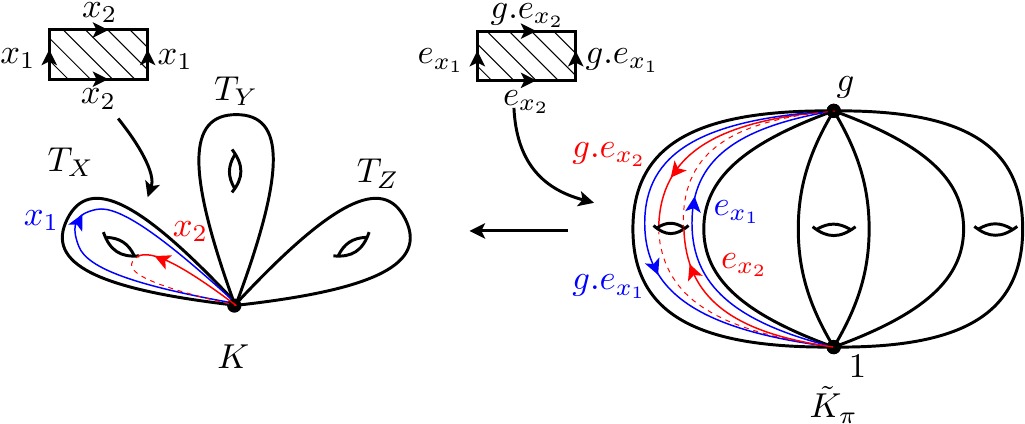}
          \caption{The $2$-fold cover $\Tilde K_\pi$ of the Salvetti complex $K$, and the lift of a $2$-cell.}
          \label{fig_tori}
        \end{figure}

	Denote by $C_i$ the $\bbQ$--vector space of $i$--dimensional chains in $\Tilde K_\pi$,
	by $C_i^{(+1)}$ and $C_i^{(-1)}$ the corresponding eigenspaces for the action of $\bbZ_2$.
	We denote by $\partial_i:C_i\ra C_{i-1}$ the boundary morphism.
	It commutes with the action of $\bbZ_2$, so by restriction we get morphisms
	$\partial_i^{(+1)}:C_{i}^{(+1)}\ra C_{i-1}^{(+1)}$ and
	$\partial_i^{(-1)}:C_{i}^{(-1)}\ra C_{i-1}^{(-1)}$.
    This decomposition of the space of chains gives the decomposition of $H_1(\Tilde K_\pi;\bbQ)$ into eigenspaces,
	and $V_{-1}$ is the quotient $\ker \partial_1^{(-1)} / \partial_2^{(-1)}(C_2^{(-1)})$.
	A basis of $C_1^{(-1)}$ is given by $\{(1-g)e_{u}\mid u\in S\}$, so its dimension is $a+b+c$.
	We have that $\partial_1(C_1)=C_0^{(-1)}\simeq\bbQ$,
	so $\ker \partial_1^{(-1)}$ has codimension $1$, i.e.\
        $\dim ( \ker \partial_1^{(-1)} ) = a+b+c - 1$.
	The image $\partial_2(C_2)$ is the subspace of $C_1$  generated by
	$(1-g)(e_{u}-e_v)$ for $u, v\in S$ in the same equivalence class (see Figure \ref{fig_tori}).
     All these elements are in $C_1^{-1}$ so $\partial_2(C_2)=\partial_2^{(-1)}( C_2^{(-1)})$ has dimension $a+b+c-3$.
	It follows that $\dim(V_{-1}) = 2$ and
	 the two cycles
	$$\vec_1 := (1-g)(e_x-e_z), \ \  \vec_2 := (1-g)(e_y-e_z)$$
	form a basis of $V_{-1}$.
			
	We thus get a representation  $\sigma:\Aut_\pi(\A_{\Gamma_S})\ra \PGL(V_{-1})\simeq \PGL(2,\bbQ)$.
	Since the lattice $H_1(\Tilde A;\bbZ)$ is preserved, a finite index subgroup of $\Aut_\pi(\A_{\Gamma_S})$ preserves
	$\bbZ \vec_1\oplus \bbZ \vec_2$, and thus has a representation in $\PGL(2,\bbZ)$.

	 One can check that all inner automorphisms act as $\pm\id$ on $\vec_1,\vec_2$, and hence lie in the kernel of $\sigma$.
	This therefore gives a representation
			$$\hat \sigma : \Out_\pi (\A_{\Gamma_S}) \to \PGL(2,\bbQ)$$
	where $\Out_\pi(\A_{\Gamma_S})=\Aut_\pi(\A_{\Gamma_S})/\Inn(\A_{\Gamma_S})$,
	which contains a finite index subgroup which maps into $\PGL(2,\bbZ)$.

	We claim that the partial conjugations of $\bbZ^a*\bbZ^b*\bbZ^c$ are in $\Aut_\pi(\A_{\Gamma_S})$, that their images preserve $\bbZ\vec_1 \oplus \bbZ\vec_2$, and that the group they generate in $\PGL(2,\bbZ)$ is virtually a non-abelian free group.

	Lifting a homotopy equivalence of $K$ inducing one of the partial conjugations,
	one can describe their action  on $V_{-1}$ as follows:
	$$C_X^{y_j}(e_{x}) = (1-g)e_{y_j} +ge_{x} = (1-g)e_y +ge_x,\ C_X^{y_j}(e_{y})=e_{y}, \text{ and } C_X^{y_j}(e_{z})=e_{z}.$$
	Using relations like $(1-g)(ge_x) = (1-g)(-e_x)$, we therefore get
$$C_X^{y_j}(\vec_1) = (1-g)(ge_x +(1-g)e_y-e_z) = -\vec_1 + 2\vec_2$$
	and 			$$C_X^{y_j}(\vec_2) = (1-g)(e_y-e_z) = \vec_2$$	
	so 	$$\hat \sigma(C_X^{y_j})=\left(\begin{array}{rr}
	-1 & 0 \\
	2 & 1
	\end{array}\right).$$
	With similar calculations we see that
	$$\hat \sigma(C_Y^{z_k})=\left(\begin{array}{rr}
	1 & 0 \\
	0 & -1
	\end{array}\right),\ \hat \sigma(C_Z^{x_i})=\left(\begin{array}{rr}
	-1 & -2 \\
	0 & 1
	\end{array}\right).$$
	We thus have
	$$\hat \sigma(C_X^yC_Y^z)=
	\left(\begin{array}{rr}
	-1 & 0 \\
	2 & -1
	\end{array}\right), \ \hat \sigma(C_Y^zC_Z^x)=\left(\begin{array}{rr}
	-1 & -2 \\
	0 & -1
	\end{array}\right).$$
	In particular the image of $\<C_X^y,C_Y^z,C_Z^x\>$ under $\hat{\sigma}$ in $\PGL(2,\bbZ)$ contains a non-abelian free group.
	The lemma follows.
\end{proof}

\subsubsection{Case 2: when $C_{X}^{z_k} \notin \cals$}
In this case, $\cals$ contains no partial conjugation of the form $C_{X}^{z_k}$ (\ie in \ref{it_pc}).
The representation used in case 1 has virtually abelian image, and does not help us.

Consider the subgroup $U=\grp{X} \ast \grp{Y}\subset \A_{\Gamma_S}$.
Since all generators in \ref{it_sil}, \ref{it_int}, \ref{it_yx}--\ref{it_zy} preserve the conjugacy class of $U$,
so does every automorphism in $\calo=\grp{\cals}$.
By restricting to $U$, we get a homomorphism $\calo\ra \Out(U)$ as in Section \ref{sec_res_fact}.

In Lemma \ref{lem:H gen set} we will show that this homomorphism lifts to a homomorphism $\calo\ra \Aut(U)$.
Before proving this,
we note the following easy, but useful, fact.

\begin{lem}\label{lem_commute}
	Consider $u\leq v$, and assume that $[u]$ is abelian with $\card{[u]}\geq 2$.
	
	Then $u$ commutes with $v$.
\end{lem}

\begin{proof}
	Consider  $u'\sim u$ with $u'\neq u$. Since $[u]$ is abelian, $u\in \lk(u')\subset \st(v)$, so $u$ and $v$ commute.
\end{proof}

Note that in the following, the notation for the generators from Section \ref{sec:prelim-generators} is used to represent both outer automorphisms (when they are in $\calo$) and honest automorphisms (when they are in, for example, $\cala_1$).

\begin{lem}\label{lem:H gen set}
	There is a homomorphism $\rho_1 : \calo \to \Aut(U)$, with image $\cala_1$ generated by

	\begin{itemize}
		\item $\left\{\ad_{x_i},\ad_{y_j}, R_{x_i}^{x_{i'}},L_{x_i}^{x_{i'}}, R_{y_j}^{y_{j'}},L_{y_j}^{y_{j'}} \colon 1\leq i,i' \leq a, 1\leq j,j' \leq b\right\}$ if $y \not\leq x$,
		\item $\left\{\ad_{x_i},\ad_{y}, R_{x_i}^{x_{i'}},L_{x_i}^{x_{i'}}, L_y^{x_i}, R_y^{x_i} \colon 1 \leq i,i'\leq a \right\}$ if $y\leq x$.
	\end{itemize}
\end{lem}

\begin{proof}
We first consider the case where $z\not\leq x$ and $z\not \leq y$, so that $\cals$ contains no generator in \ref{it_zx}, \ref{it_zy}.
Since the other generators preserve the conjugacy class of $\grp{Z}$,
applying Lemma \ref{lem_usable_amalg} to the decomposition $\Gamma_S=U*\grp{Z}$, we get an embedding of $\calo$ into $\Aut(U)\times \Aut(Z)$,
and hence, by projection, a homomorphism $\rho_1:\calo\ra \Aut(U)$.
As mentionned in Lemma \ref{lem_usable_amalg}, given $\Phi\in \calo\subset \Out(\Gamma_S)$,
$\rho_1(\Phi)$ is defined by choosing the representative
$\Tilde\Phi\in \Aut(\Gamma_S)$ that preserves simultaneously $U$ and $\grp{Z}$, and taking for $\rho_1(\Phi)$ the restriction of $\Tilde \Phi$ to $U$.

For a SIL automorphism $\Phi=[C_Z^{x_i}]\in \Out(\Gamma_S)$  in \ref{it_sil},
one gets $\Tilde \Phi=\ad_{x_i}\m\circ C_Z^{x_i}$ so $\rho_1(\Phi)=\ad_{x_i}\m \in \Aut(U)$.
Similarly, $\rho_1([C_Z^{y_j}])=\ad_{y_j}\m$.
The internal transvections in \ref{it_int} readily preserve $U$ and $\grp{Z}$, and their restrictions to $U$ give
the automorphisms $R_{x_i}^{x_{i'}},L_{x_i}^{x_{i'}}, R_{y_j}^{y_{j'}},L_{y_j}^{y_{j'}}$ (the internal transvections on $Z$ restrict to the identity).

If $y\not \leq x$, then there is no other generator in $\cals$ and we are done.
If $y\leq x$, then $Y=\{ y \}$ by Lemma \ref{lem_commute}, so there are no internal transvections on $Y$.
The transvections $L_y^{x_i}$ and $R_y^{x_i}$ in \ref{it_yx} preserve $U$ and $\grp{Z}$, so their image under $\rho_1$
are $L_y^{x_i}$ and $R_y^{x_i}$.

There remains to consider the case where $z\leq x$ or $z\leq y$.
In either situation, the equivalence class $Z$ is reduced to $\{z\}$ by Lemma \ref{lem_commute},
so $\A_{\Gamma_S}$ can be written as an HNN extension $\A_{\Gamma_S}=U*_{\{1\}}$ with stable letter $z$.
The first part of Lemma \ref{lem_usable_HNN} then yields a homomorphism $\calo\ra \Aut(U)\semidirect U$, hence a homomorphism $\rho_1:\calo\ra \Aut(U)$,
defined as follows:
given $\Phi\in \calo\subset \Out(\Gamma_S)$, choose
$\Tilde\Phi\in \Aut(\A_{\Gamma_S})$ that preserves $U$ and sends $z$ into $zU$, then $\rho_1(\Phi)=\Tilde\Phi_{|U}$.

One checks as above that the image of the SIL automorphisms in \ref{it_sil} under $\rho_1$ are the inner automorphisms $\ad_{x_i}\m,\ad_{y_j}\m$,
and that internal transvections in \ref{it_int} yield all the transvections $R_{x_i}^{x_{i'}}$, $L_{x_i}^{x_{i'}}$, $R_{y_j}^{y_{j'}}$, and $L_{y_j}^{y_{j'}}$.
The right transvections $R_{z}^{x_i}$, $R_{z}^{y_j}$ in \ref{it_zx}, \ref{it_zy} are mapped to the identity
while the left transvections $L_{z}^{x_i}$, $L_{z}^{y_j}$ are mapped to $\ad_{x_i}\m$, $\ad_{y_j}\m$ respectively
(this is because $\ad_{x_i}\m\circ L_{z}^{x_i}$ is the representative that sends $z$ to $zx_i\in zU$).
This all applies when  $y\not \leq x$.
If $y\leq x$, then as above $Y=\{y\}$ and the transvections $L_y^{x_i}$ and $R_y^{x_i}$ are mapped under $\rho_1$ to $L_y^{x_i}$ and $R_y^{x_i}$ in $\Aut(U)$.
\end{proof}

Before treating the general case, let's first give a simpler argument showing largeness of $\cala_1$
in the particular case where $a=b=1$ (this argument is not needed as the argument below will include this case).
If $y\not \leq x$, we see from Lemma \ref{lem:H gen set} that  $\cala_1=\grp{\ad_{x},\ad_y}$.
Since $x,y$ do not commute, this is isomorphic to a free group, giving largeness in this instance.
If $y \leq x$, then $\cala_1 = \grp{\ad_x,\ad_y,L_y^x,R_y^x}$.
Writing  $x=\ad_x,y=\ad_y,L=L_y^x,R=R_y^x$,
since $R=x\m L$, $\cala_1$ is isomorphic to the semidirect product $\grp{L}\semidirect \grp{x,y}$
which can be presented as
$$\grp{x,y,L\mid [L,x]=1, LyL\m=xy}.$$
Since this presentation has deficiency one and includes a commutator,  and since $\grp{x,L}$ has infinite index in the abelianisation of $\cala_1$,
this group is large by a result of Button \cite[Theorem 3.1]{Button_large_deficiency}.

\begin{lem}\label{lem_iso}
  Let $\cala_1$ be  the subgroup of $\Aut(U)$ given in Lemma \ref{lem:H gen set}, where $U=\grp{X}*\grp{Y}\simeq \bbZ^a*\bbZ^b$.

Then one of the following  holds:
\begin{itemize}
\item either $y\not \leq x$, and $\cala_1$ splits as an amalgam $\cala_1\simeq \Big[\calo_1\semidirect \bbZ^a\Big] *_{\calo_1} \Big[ \calo_1\semidirect \bbZ^b\Big]$
with $\calo_1=\SL_a(\bbZ)\times \SL_b(\bbZ)$ acting on $\bbZ^a$ and on $\bbZ^b$ in the natural way;
\item or $y\leq x$, $b=1$ and $\cala_1$ splits as an HNN extension
$\cala_1\simeq \Big[\SL_a(\bbZ)\semidirect (\bbZ^a\times \bbZ^a)\Big] *_{\calo_1} $ where $\calo_1=\SL_a(\bbZ)\semidirect \bbZ^a$, and the two
embeddings of $\calo_1$ defining the HNN extension are the natural maps $j_1,j_2$ defined by
$$j_1:\calo_1\simeq \SL_a(\bbZ)\semidirect (\bbZ^a\times \{1\})\hookrightarrow \SL_a(\bbZ)\semidirect (\bbZ^a\times \bbZ^a)$$
and
$$j_2:\calo_1\simeq \SL_a(\bbZ)\semidirect (\{1\} \times \bbZ^a)\hookrightarrow \SL_a(\bbZ)\semidirect (\bbZ^a\times \bbZ^a).$$
\end{itemize}
\end{lem}

 We give the proof of Lemma \ref{lem_iso} after the following corollary, which will conclude the proof of Lemma \ref{lem:special SIL implies large}.

\begin{cor}\label{cor_large}
  The group $\cala_1$ is large, and hence so are $\calo$ and $\Out(\A_\Gamma)$.
\end{cor}

\begin{proof}[Proof of Corollary \ref{cor_large}.]
Assume first that $\cala_1=  \Big[\calo_1\semidirect \bbZ^a\Big] *_{\calo_1} \Big[ \calo_1\semidirect \bbZ^b\Big]$
with $\calo_1=\SL_a(\bbZ)\times \SL_b(\bbZ)$.
Consider $\bar \calo_1=\SL_a(\bbZ_3)\times \SL_b(\bbZ_3)$, a finite quotient of $\calo_1$.
Then define  $$\bar \cala_1=\Big[\bar\calo_1\semidirect \bbZ_3^a\Big] *_{\bar \calo_1} \Big[ \bar \calo_1\semidirect \bbZ_3^b\Big].$$
This quotient of $\cala_1$ is virtually free, and not virtually cyclic (even if $a=1$ or $b=1$).
This implies that $\cala_1$ is large.
The case of an HNN extension is similar.
\end{proof}

\begin{proof}[Proof of Lemma \ref{lem_iso}.]
We denote by $\calo_1$ the image of $\cala_1$ in $\Out(U)$,  and show that it is $\SL_a(\bbZ)\times\SL_b(\bbZ)$, or $\SL_a(\bbZ)\ltimes \bbZ^a$ respectively.

First consider the case where $y\not\leq x$, so that $\cala_1$ has a generating set
given by the first possibility in Lemma  \ref{lem:H gen set}.
The first assertion of Lemma \ref{lem_usable_amalg} shows that $\calo_1\hookrightarrow \Aut(\bbZ^a)\times \Aut(\bbZ^b)$
and looking at the image of the generating set of $\cala_1$,
we see that the image of $\calo_1$ is precisely $\SL_a(\bbZ)\times\SL_b(\bbZ)$.
The second assertion of Lemma \ref{lem_usable_amalg} shows that $\cala_1$ splits as an amalgam
$$\cala_1\simeq \Big [\calo_1\semidirect \bbZ^a\Big] *_{\calo_1} \Big[\calo_1\semidirect \bbZ^b\Big].$$
Since the actions are the natural ones,
this concludes the lemma in this case.

In the case where $y\leq x$,  $Y=\{y\}$ by Lemma \ref{lem_commute}, so $U$ splits as an HNN extension $U=\grp{X}*_{\{1\}}$,
and $y$ is a stable letter.
The group $\cala_1$ has a generating set given by the second possibility in Lemma \ref{lem:H gen set}.
Let $\cala'\subset \Aut(\grp{X})$ be the image of $\cala_1$ under the natural map to $\Out(\grp{X})=\Aut(\grp{X})$.
In view of the generating set, we see that $\cala'\simeq \SL_a(\bbZ)$.
Then the second part of Lemma \ref{lem_usable_HNN}
  applies to $\cala_1$, and says that $\cala_1$ splits as an HNN extension
\begin{eqnarray*}\cala_1&=
	&  \Big[\cala' \semidirect (\grp{X}\times \grp{X})\Big]*_{\cala'\semidirect \grp{X}} \\
&\simeq & \Big[\SL_a(\bbZ) \semidirect (\bbZ^a\times \bbZ^a)\Big]*_{\SL_a(\bbZ)\semidirect \bbZ^a}.
\end{eqnarray*}
The lemma follows.
\end{proof}

\section{A short exact sequence when there is no SIL}\label{sec:SES when no SIL}

In this section, we prove the third alternative of the trichotomy.

\begin{prop}\label{prop_tricho3}
	Let $\Gamma$ be a finite graph  with no SIL.
	
	Then $\Out^0(\A_\Gamma)$ fits in a short exact sequence
	$$1\ra P\ra \Out^0(\A_\Gamma) \ra \prod_{i=1}^k \SL_{n_i}(\bbZ)\ra 1$$
	where $P$ is finitely generated nilpotent, and $n_1,\dots, n_k$ are the sizes of the equivalence classes in $\Gamma$.
\end{prop}

Since any free equivalence class of size 3 or more yields a SIL, the assumptions of the Proposition imply that
free equivalence classes have size at most two.
The following lemma shows that the absence of SILs also puts strong constraints on these free equivalence classes.
It won't be needed in the sequel.

\begin{lem}\label{lem_product}
	If $\Gamma$ has no SIL then $\A_\Gamma = (\bbF_2)^k \times \A_\Lambda$,
	where $\Lambda$ is a graph with no non-abelian free equivalence class
	and $k\geq 0$ is the number of non-abelian free equivalence classes in $\Gamma$.
\end{lem}

\begin{proof}
	We proceed by induction on $k$. The case $k=0$ is immediate.
	Suppose we have $x \sim y$, with $[x,y] \neq 1$.
	Let $\Gamma'=\lk(x)=\lk(y)$. Any vertex $z \in \Gamma \setminus (\{x,y\}\cup \Gamma')$, yields a SIL $(x,y\mids z)$ so
	$\Gamma=\Gamma'\cup\{x,y\}$, and $A_\Gamma=\bbF_2\times A_{\Gamma'}$.
	
	Since for any $v\in \Gamma'$, $\lk_{\Gamma}(v)=\lk_{\Gamma'}(v)\cup\{x,y\}$,
	two vertices of $\Gamma'$ are equivalent in $\Gamma'$ if and only if they are in $\Gamma$.
	Thus $\Gamma'$ has $k-1$ non-abelian free equivalence classes so by induction,
	$\A_{\Gamma'} = (\bbF_2)^{k-1} \times \A_{\Lambda}$, with $\Lambda$ as in the lemma.
	This completes the proof.
\end{proof}

We will make use of the standard representation of $\Out(\A_\Gamma)$, see Section \ref{sec:standard representation}. Recall that its kernel is denoted by $\IA(\A_\Gamma)$.
We will show that having no SIL implies that $\IA(\A_\Gamma)$ is abelian.
As a first step towards this, we give the following simple lemma.

\begin{lem}\label{lem:no SIL dominating pair}
	Suppose $\Gamma$ contains no SIL, and there are distinct vertices $v,x,y$ such that $v\leq x,y$.
	
	Then $[x,y]=1$.
\end{lem}

\begin{proof}
	If $x$ commutes with $v$, then $x\in \lk(v)\subset \st(y)$, so $[x,y]=1$.
	So we can assume that $x,y\notin \st(v)$.
	This therefore implies that $\lk(v)\subset \lk(x)\cap \lk(y)$, so in particular $(x,y\mids v)$ will form a SIL whenever $[x,y]\neq 1$.
\end{proof}

\begin{prop}\label{prop:IA abelian}
	Suppose $\Gamma$ has no SIL.
	
	Then $\IA(\A_\Gamma)$ is abelian.
\end{prop}

\begin{proof}
	
	A generating set $\calm_\Gamma$ for $\IA(\A_\Gamma)$ of Day and Wade is given in Section \ref{sec:torelli}. It consists of
	all partial conjugations,
	along with commutator transvections
	$$K_v^{[x,y]}:v\mapsto v[x,y], \ \ \textrm{ when } v\leq x, y,\ v\notin \{x,y\}.$$
	However, by Lemma \ref{lem:no SIL dominating pair}, all commutator transvections are trivial when there is no SIL.
	
	There remains to check that
	the group generated by all partial conjugations is abelian.
	This is due to \cite[Thm. 3.8]{CRSV_no_SIL} and \cite[Thm. 1.4]{GPR_automorphisms},  and also follows from \cite[Lemma 2.5]{Day_solvable}.
	We give a simple self-contained proof.
	
	\begin{figure}[ht!]
		\labellist \small \hair 4pt
		\pinlabel $\st(x_1)$ at 250 345
		\pinlabel $\st(x_2)$ at 35 185
		\pinlabel $\lk(x_1)\cap\lk(x_2)$ at 251 163
		\pinlabel $Z_1\cap Z_2$ at 450 35
		\pinlabel $Z_1$ [t] at 405 0
		\pinlabel $Z_2$ [l] at 510 77.5
		\pinlabel $z$ [l] at 400 79
		\pinlabel $p_1$ [l] at 415 105
		\pinlabel $p_j$ [l] at 370 184
		\pinlabel $p_{n-1}$ [t] at 292 303
		\pinlabel $x_1$ [t] at 250 315
		\pinlabel $x_2$ [b] at 125 185
		\endlabellist
		
		\centering
		\includegraphics[width=10cm]{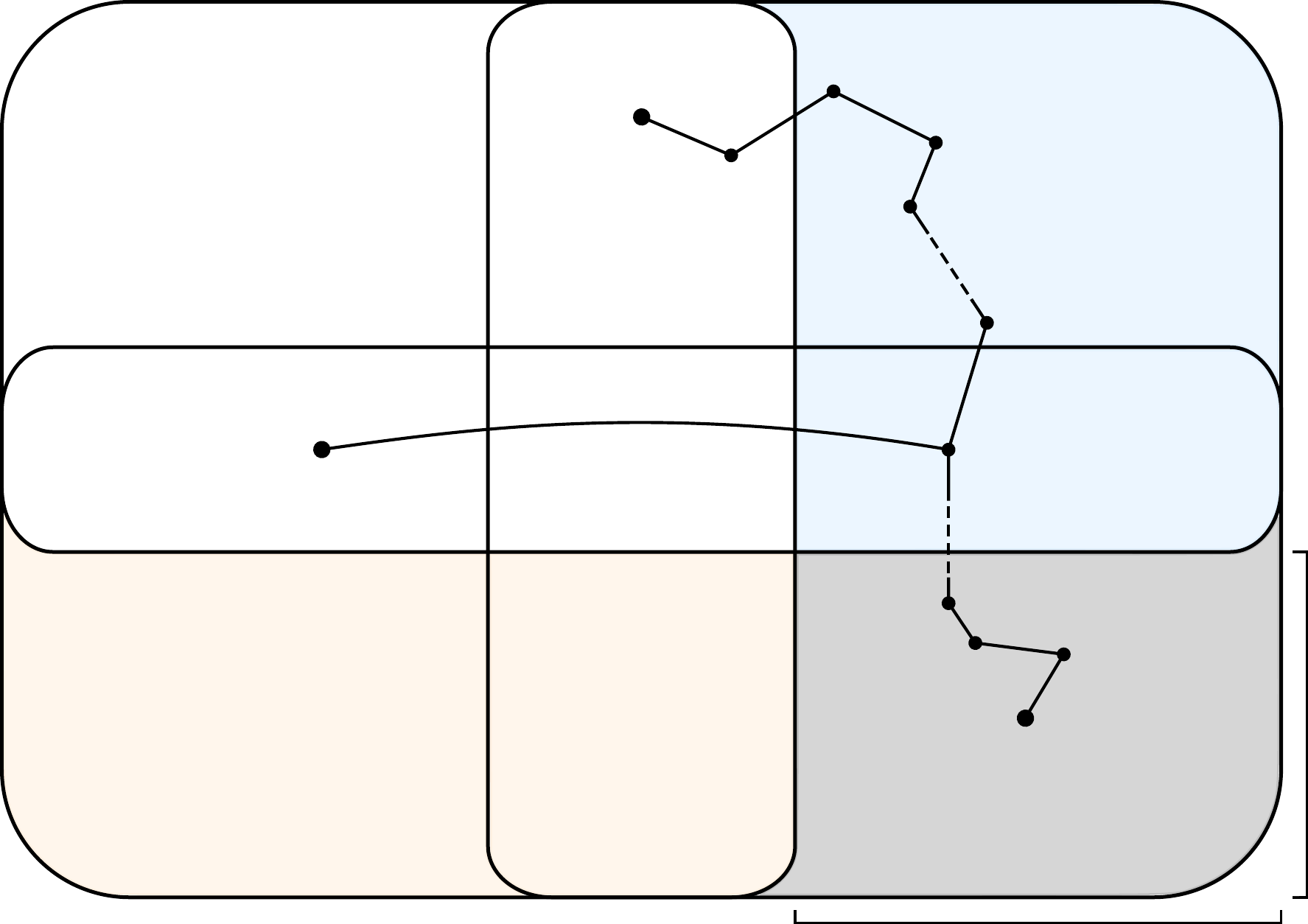}
		\vspace{3mm}
		\caption{If the partial conjugations $C_{Z_1}^{x_1}$ and $C_{Z_1}^{x_1}$ do not commute, then any path from $Z_1\cap Z_2$ to $\{x_1,x_2\}$ passes through $\lk(x_1)\cap\lk(x_2)$, forming a SIL.}\label{fig:no sil pc abelian}
	\end{figure}
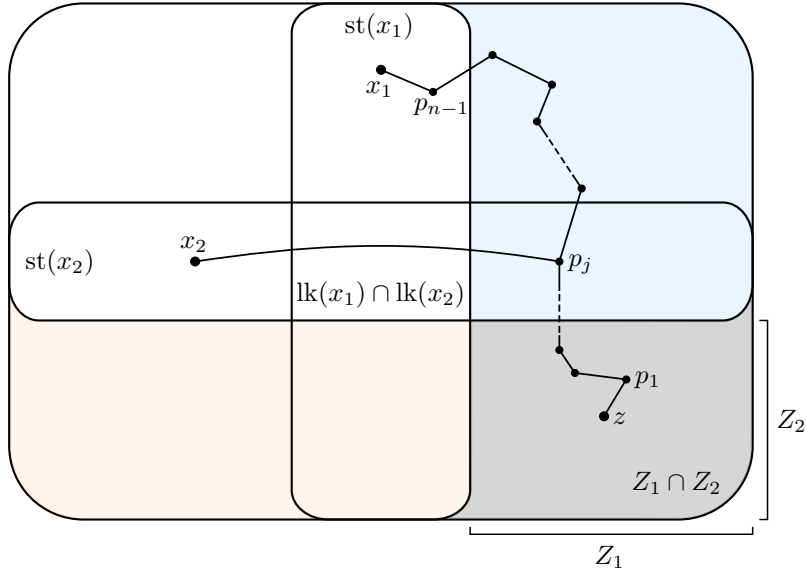
	
	Let $C_{Z_1}^{x_1}$, $C_{Z_2}^{x_2}$ be two partial conjugations where $Z_i$ is a union of connected components of $\Gamma\setminus \st(x_i)$,
	and let's prove that they commute in $\Out(\A_\Gamma)$.
	If one changes $Z_i$ to $Z'_i=\Gamma\setminus (\st(x_i)\cup Z_i)$, then $C_{Z'_i}^{x_i}=(C_{Z_i}^{x_i})\m$ in $\Out(\A_\Gamma)$.
	Thus, up to changing $Z_i$ to $Z'_i$, one may assume that $x_1\notin Z_2$ and $x_2\notin Z_1$.
	If $Z_1\cap Z_2=\es$, then $C_{Z_1}^{x_1}$, $C_{Z_2}^{x_2}$ commute.
	If $x_1$ and $x_2$ commute then $C_{Z_1}^{x_1}$, $C_{Z_2}^{x_2}$ also commute.
	
	So assume that $x_1,x_2$ don't commute and consider $z\in Z_1\cap Z_2$.
	We claim that $(x_1,x_2\mids z)$ is a SIL, thus concluding the proof.
	If not, then there is a path joining $z$ to $x_1$ or $x_2$ which does not meet $\lk(x_1)\cap \lk(x_2)$.
	Choose a shortest such path $p$, and up to exchanging the roles of $x_1$ and $x_2$, write this path as
	$z=p_0,p_1,\dots,p_n=x_1$, see Figure \ref{fig:no sil pc abelian}.
	Since $z\in Z_2$ and $x_1\notin Z_2$, the path $p$ has to meet $\st(x_2)$, and therefore $\lk(x_2)$.
	Let $j$ be the smallest index such that $p_j\in \lk(x_2)$.
	We have $j\neq n$ because $x_2$ does not commute with $x_1$, and $j\neq n-1$ since otherwise, $p_{n-1}\in \lk(x_1)\cap \lk(x_2)$.
	Thus, the path $z=p_0,p_1,\dots,p_j,x_2$ is a shorter path which contradicts the choice of $p$.
\end{proof}

The standard representation yields the short exact sequence
$$1 \ra \IA(\A_\Gamma) \ra \Out^0(\A_\Gamma) \ra \calg_\bbZ \ra 1$$
where $\calg_\bbZ$ is a subgroup of the group of  block  lower-triangular unipotent matrices with integer entries, where each
block corresponds to an equivalence class (see Section \ref{sec:standard representation}).
Denote by $n_1,\dots,n_k$ the sizes of the equivalence classes.
Then $\calg_\bbZ$ maps onto  $\prod_{i=1}^k \SL_{n_i}(\bbZ)$ with kernel $\calu$ consisting of lower-triangular unipotent matrices.
In particular, $\calu$ is nilpotent and,
since $\IA(\A_\Gamma)$ is abelian by Proposition \ref{prop:IA abelian}, its preimage  $P$ in  $\Out^0(\A_\Gamma)$ is abelian-by-nilpotent.
We thus get the short exact sequence
$$1 \ra P \ra \Out^0(\A_\Gamma) \ra \prod_{i=1}^k \SL_{n_i}(\bbZ) \ra 1.$$
It remains to prove that $P$ is in fact nilpotent. To do so, we first verify it has a nice generating set.

\begin{lem}\label{lem:P gen set}
	The subgroup $P$ is generated by partial conjugations and transvections $R^u_v$ for $u\geq v$, $u\not\sim v$.
\end{lem}

\begin{proof}
	Let $q$ denote the projection map from $\calg_\bbZ$ onto $\prod_{i=1}^k \SL_{n_i}(\bbZ)$.
	Then $P$ is the kernel of $q\circ \sigma$, where $\sigma$ is the standard representation.
	Since $\ker(q)$ is the group generated by the elementary matrices $E_{u,v}$ with $u\geq v$ but $u\not \sim v$,
	$P$ is generated by $\IA(\A_\Gamma)$ and by the transvections $R^u_v$ with $u\geq v$ but $u\not \sim v$.
	Since, by Lemma \ref{lem:no SIL dominating pair}, $\IA(\A_\Gamma)$ is generated by partial conjugations, the lemma follows.
\end{proof}

The final step in proving Proposition \ref{prop_tricho3} is the following.
The result can in fact be inferred from Day's proof that having no SIL and no equivalence class of size 2 or more implies that $\Out^0(\A_\Gamma)$ is nilpotent \cite[Prop. 2.11]{Day_solvable}.
We give the relevant part here for completeness.

\begin{lem}
	The subgroup $P$ is nilpotent.
\end{lem}

\begin{proof}
	In view of Lemma \ref{lem:P gen set}, we need to understand commutators of transvections and partial conjugations, and their inverses.
	We note though that the inverse in $\Out(\A_\Gamma)$ of a partial conjugation  $C^y_Y$ is also a partial conjugation, namely $C^{y}_{Y'}$ where $Y'=\Gamma\setminus (Y \cup \st(y))$, so these do not require special attention.
	
	By \cite[Lemma 2.5]{Day_solvable}, having no SIL implies that any two generators (partial conjugation or transvection)
	that fix each other's multiplier commute.
	For any $x\in\Gamma$, any partial conjugation $C^w_W$ has a representative in $\Aut(\A_\Gamma)$ that fixes $x$.
	This shows that partial conjugations commute with each other and with transvections fixing their multiplier.

	For the remaining pairs of generators, we have the following relations,
	that were proved by Day in the case when there is no SIL and all equivalence classes have size one \cite[Lemma 2.7]{Day_solvable}.
	\begin{eqnarray}
	\label{eq:no SIL relator 1}
	[R^x_y, C^{y}_Y]&=&C^{x}_Y \text{ for $x\geq y$, $x\not \sim y$}\\
	\label{eq:no SIL relator 2}
	[R^x_y, (R^{y}_z)^{\pm1}]& =& (R^{x}_z)^{\pm1}  \text{ for $x\geq y\geq z$ with $x\not\sim y\not \sim z$.}
	\end{eqnarray}
	
	In \eqref{eq:no SIL relator 1}, $Y$ is assumed to be a union of connected components of $\Gamma \setminus \st(y)$.
	Since $x\geq y$, $Y$ is also a union of connected components of $\Gamma \setminus \st(x)$, together with elements of $\st(x)$ (on which conjugation by $x$ has no effect), so $C_Y^x$ is well defined.
	To prove these relations,
	we claim that $[x,y]=1$ in both cases.
	The relations  \eqref{eq:no SIL relator 1}--\eqref{eq:no SIL relator 2} then easily follow by a direct computation (left to the reader).

	In \eqref{eq:no SIL relator 2}, the fact that $[x,y]=1$ follows from Lemma \ref{lem:no SIL dominating pair}.
	The relation \eqref{eq:no SIL relator 1} is clear if $C_Y^y$ is inner because then so is $C_Y^x$.
	So we may assume that $C_Y^y$ is not inner
	and that $\st(y)$ disconnects $\Gamma$ into two different non-empty sets $Y$ and $Y'=\Gamma\setminus (Y \cup \st(y))$.
	Without loss of generality, we can assume that $x \in Y$. Assuming for contradiction that $[x,y]\neq 1$, since $\lk(x)\supset \lk(y)$,
	we get that any $z\in Y'$ yields a SIL $(x,y \mids z)$, a contradiction.
	
	We now verify these relations imply that $P$ is nilpotent.
	Given a vertex $x$, its \emph{depth}, denoted $\depth(x)$, is defined to be the maximal $m$ such that there are vertices $x_2,\ldots, x_m$ satisfying  $x=x_1\geq x_2 \geq \ldots \geq x_m$ and $x_i \not \sim x_{i+1}$ for $1\leq i <m$.
	Let $K$ be the maximal depth of vertices in $\Gamma$.
	For $1 \leq i \leq K$, let $S_i$ be the set consisting
	of partial conjugations $C^w_W$ with $\depth(w) \geq i$,
	and of transvections $R^x_y$ with $y\not\sim x$
	and $\depth(x)-\depth(y) \geq i $, $y\leq x$.
	Then $S_i \subseteq S_{i-1}$, and $S_{K+1}$ is empty.
	
	Consider $X \in S_i$ and $Y\in S_j$, for some $1\leq i,j\leq K$.
	It follows from the commutator relations  (when each others multipliers are fixed) and from \eqref{eq:no SIL relator 1} and \eqref{eq:no SIL relator 2}
	that  $[X,Y^{\pm 1}] \in \grp{S_{i+j}}$.
	Thus it follows that $P$ is nilpotent.
\end{proof}

\section{McCool groups}\label{sec:mccool}

Let $\bbF$ be a finitely generated free group and $\calc$ a finite set of conjugacy classes of elements in $\bbF$.
For the McCool group of $\bbF$ corresponding to $\calc$, we prove the following alternative. (See Section \ref{sec:prelim McCool} for definitions).

\begin{thm}\label{thm_MC}
  Let $M=\Mc(\bbF,\calc)$ be a McCool group
  and $\bbF = H_1 * \dots * H_k * \bbF_r$ be the Grushko decomposition of $\bbF$ relative to $\calc$.
  Then one of the following holds. 
  \begin{enumerate}
 \renewcommand{\theenumi}{($\calm$\arabic{enumi})}
  \item \label{it_Fr} If $r\geq 2$, then $M$ maps onto $\Out(\bbF_r)$.
  \item \label{it_FF} 
If either  $k=2$ and $r=1$, or $k\geq 3$, then $M$ is large.
  \item \label{it_sporadic} If $k+r=2$ and $r\leq 1$,
    then either $M$ is large, or it maps onto a non-virtually abelian mapping class group of a (maybe non-orientable) hyperbolic surface with at least one puncture, or $M$ is virtually abelian.
  \item \label{it_1bout} If $k+r=1$,
    then some finite index subgroup  $M_0$ of $M$ fits into a short exact sequence
$$1\ra \calt \ra M_0 \ra \prod_{j=1}^s \MCG(S_j)\ra 1$$
where $\calt$ is finitely generated and abelian, and $\MCG(S_j)$ is the mapping class group of a (maybe non-orientable) hyperbolic surface with at least one puncture.   \end{enumerate}
\end{thm}

\begin{rem}\label{rem:MCG}
In the last assertion, it may happen that $s=0$, in which case $G$ is virtually abelian.
The mapping class group $\MCG(S_j)$ mentioned in this assertion is the group of homeomorphisms that preserve each boundary component
and its orientation,
modulo the connected component of the identity.
This can be viewed as the pure mapping class group of a punctured surface (fixing each puncture, and preserving the local orientation at each puncture).
\end{rem}

In some sporadic cases, the mapping class groups appearing in Theorem \ref{thm_MC} above may happen to be virtually abelian.
The following proposition gathers some known results showing that they are otherwise vast.

\begin{prop}\label{prop:v abelian mcg}
	Let $S$ be a hyperbolic (maybe non-orientable) surface, with at least one puncture.
	Then:
	\begin{enumerate}[(i)]
		\item $\MCG(S)$ is finite (or trivial) if and only if $S$ is
		a sphere with at most three punctures, or
		a projective plane with at most two punctures;
		
		\item $\MCG(S)$ is infinite virtually abelian if and only if $S$ is a once punctured Klein bottle;
		
		\item  In all other cases, $\MCG(S)$ involves all finite groups, is SQ-universal, virtually has many quasimorphisms, and is not boundedly generated.
	\end{enumerate}	
\end{prop}

\begin{proof}
The mapping class group of a sphere with three punctures is trivial.
	A projective plane with two punctures has finite mapping class group \cite[Cor. 4.6]{Korkmaz_MCG}.
	For the once punctured Klein bottle, its mapping class group is virtually cyclic \cite[Thm. A.5]{Stukow}.

When $S$ is not as in (i) or (ii), then it contains
two simple closed curves $\alpha,\beta$ not bounding a disk, a punctured disk or a M\"obius band,
and with non-zero geometric intersection number.
By  \cite[Prop. 4.7]{Stukow},
the corresponding Dehn twists $t_\alpha,t_\beta$ don't have commuting powers, so $\MCG(S)$ is not virtually abelian.

Now suppose $\MCG(S)$ is not virtually abelian.
Bestvina--Fujiwara proved that they admit an action on a hyperbolic space with WPD elements \cite{BeFu_cohomology,BeFu_quasi}
and deduce that they have many quasimorphisms (and hence are not boundedly generated).
SQ-universality follows by \cite{DGO_HE}.

By Masbaum and Reid \cite{MaRe_finite}, (see also \cite{GLLM_arithmetic_mcg}), mapping class
groups of {closed} orientable surfaces of genus at least one involve all finite groups.
We give below a (different) proof that this is the case for
non-virtually abelian mapping class groups of hyperbolic surfaces
with at least one puncture, orientable or not.

With our notations,
$\MCG(S)$ is a finite index subgroup of the pure mapping class group $G=\calp\calm(S)$
of homeomorphisms fixing all the marked points (with no condition on the orientation).
It is enough to prove that $G$ is virtually abelian or involves all finite groups.

Let $S'$ be obtained from $S$ by forgetting one marked point.
We use the Birman exact sequence
$$\pi_1(S')\xra[\delta]{} G \ra G'\ra 1$$
where $G'$ is the pure mapping class group of $S'$ (\cite{Birman_MCG}, see \cite{Korkmaz_MCG} for the non-orientable case).

If $S'$ has negative Euler characteristic then $\pi_1(S')$ is large.
If moreover $\delta$ is injective,
then $\pi_1(S')$ is a large finitely generated normal subgroup of $G$,
so $G$ involves all finite groups by Lemma \ref{lem:technical_lemma}.

It may happen that $\delta$ is not injective (e.g. if $S$ is a punctured torus).
By  \cite[Th.1]{Birman_MCG}, \cite{Korkmaz_MCG}, the kernel of $\delta$ is the image of $\pi_1(\operatorname{Homeo}(S'))$.
Since the fundamental group of any topological group is abelian, the kernel of $\delta$ is abelian.
But if $S'$ has negative Euler characteristic,
its fundamental group has no non-trivial abelian normal subgroup, so $\delta$ is injective, and the argument applies.

By parts (i) and (ii), the only remaining case to consider is
when $S$ is a once punctured torus. Then $\MCG(S)\simeq \SL_2(\bbZ)$ and it involves all finite groups.
\end{proof}

\begin{proof}[Proof of Theorem \ref{thm_MC}]
If $k+r=1$, then $\bbF$ is freely indecomposable relative to $\calc$ (the case where $\bbF=\bbZ$ being trivial).
By Theorem \ref{thm:mccool ses} \cite[Th 4.6]{GL6}, $M$ has a finite index subgroup that fits in a short exact sequence as in assertion \ref{it_1bout}.

Since $M$ preserves the conjugacy class of each  $c\in\calc$, it follows that the conjugacy class of each $H_i$ is also preserved under any automorphism in $M$.
Hence the normal subgroup of $\bbF$ generated by $H_1,\dots,H_k$ is $M$--invariant,
so induces an automorphism of the quotient group $\bbF_r$.
This yields a map $M\ra \Out(\bbF_r)$ which is clearly onto.
Thus if $r\geq 2$, assertion \ref{it_Fr} holds.

Now assume that either $r=1$ and $k=2$, or $k\geq 3$ and $r\geq 0$.
Since the conjugacy class of each factor $H_i$ is $M$--invariant, we may construct
a map from $M$ to $\Out(H_1*H_2*H_3)$ or $\Out(H_1*H_2*\bbZ)$ (using the former if  $k\geq 3$ and the latter otherwise) by taking the quotient of $\bbF$ by the normal closure of the unused $H_i$ factors  (if there are any).
The image of $M$ still preserves the conjugacy classes of $H_1,H_2$ (and $H_3$ if  $k\geq 3$).
Abelianising each $H_i$ yields a map from $M$ to $\Out(\bbZ^{n_1}*\bbZ^{n_2}*\bbZ^{n_3})$, with $n_i\geq 1$ and $n_3=1$ if  $k=2$.
Note that for each $i\neq j$ and each $h\in H_i$, $M$ contains the partial conjugation that restricts to $\ad_{h}$ on $H_j$
and to the identity on all other factors (when  $k=2$ this also holds  if $i$ or $j$ corresponds to the last factor $\bbZ$).
This implies that the image of $M$ in $\Out(\bbZ^{n_1}*\bbZ^{n_2}*\bbZ^{n_3})$ contains the three partial conjugations
required in the hypotheses of Lemma \ref{lem:three partial conjugations}.
This lemma then says that the image of $M$ is large, and so assertion \ref{it_FF} holds.

There remains to consider the case where $k+r=2$ and $r\leq 1$,
ie $\bbF=H_1*H_2$ or $\bbF=H_1*\bbZ$.
Then the embedding of Lemma \ref{lem_usable_amalg} (resp.\ Lemma \ref{lem_usable_HNN}) gives
$$
\begin{array}{lll}
M \cong \Aut(H_1;\calc_{|H_1})\times \Aut(H_2;\calc_{|H_2})
&& \textrm{if $r=0$,}\\
M \cong \Aut(H_1;\calc_{|H_1})\semidirect H_1
&& \textrm{if $r=1$.}
\end{array}
$$
where $\Aut(H_i;\calc_{|H_i})$ denotes the group of automorphisms preserving each conjugacy class in $\calc_{|H_i}$
($\calc_{|H_i}$ is defined in Section \ref{sec:prelim McCool}).
If $H_1$ and $H_2$ are both cyclic  (resp.\ $H_1$ is cyclic and $r=1$), then $M$ is virtually abelian.
So assume without loss of generality that $H_1$ has rank at least $2$.
Denote by $\cala=\Aut(H_1;\calc_{|H_1})$, and $\calo$ its image in $\Out(H_1)$.
We are going to distinguish several cases, and in each case, we will
prove that  $\cala$ is large or maps onto a non-virtually abelian mapping class group.
Since in both cases $M$ maps onto $\cala$, this will show that $M$ satisfies assertion \ref{it_sporadic} and will conclude the proof.

The separate cases are identified according to the nature of the canonical cyclic JSJ decomposition $\Gamma$ of $H_1$ relative to $\calc_{|H_1}$.
If $\Gamma$ is trivial, consisting of a rigid vertex, then $\calo$ is finite by Theorem \ref{thm:mccool ses} (see Remark \ref{rem_trivial_JSJ}).
In this case, $\cala$ contains $H_1$ with finite index, so is virtually free, hence large.

If $\Gamma$ is a trivial decomposition consisting of a QH vertex whose mapping class group is finite (like a twice punctured projective plane),
$\cala$ is virtually free for the same reason.
If $\Gamma$ is a trivial decomposition consisting of a QH vertex whose mapping class group is not virtually abelian, then $\cala$ maps
onto this mapping class group and we are done.
If $\Gamma$ is a trivial decomposition consisting of a QH vertex whose mapping class group is infinite and virtually abelian,
then the underlying surface is a punctured Klein bottle
by Proposition \ref{prop:v abelian mcg}.
We'll take care of this case at the end of the argument.

Now assume that $\Gamma$ is a non-trivial decomposition. More generally, this argument will apply as soon as $H_1$ has
a non-trivial cyclic splitting relative to $\calc$ that is invariant under $\cala$ (as in Section \ref{sec:prelim usable results}).
First assume that the first Betti number of $\Gamma$ (as a plain graph) is at least $2$.
The fact that $\Gamma$ is $\cala$--invariant implies that $\cala$ acts on the Bass-Serre tree $T$ of $\Gamma$,
where the action of inner automorphisms of $H_1$ coincides with the action of $H_1$ under the natural identification.
Then let $\cala_0\subset \cala$ be the finite index subgroup acting trivially on $T/H_1$.
Then the quotient graph $T/\cala_0$ is homeomorphic to $T/H_1$.
Thus $\cala_0$ is the fundamental group of a graph of groups whose first Betti number is at least $2$,
so $\cala_0$ has an epimorphism onto $\bbF_2$.

In general, we will replace $H_1$ by a finite index subgroup, to get another splitting
with Betti number at least $2$.

We first construct a splitting whose edge group is generated by a primitive element (\ie which is part of a basis of the free group $H_1$).
Let $a\in H_1$ be the generator of the stabilizer of some edge $e$ in $T$.
By Hall's theorem \cite{Hall} there is a finite index subgroup $H'$ of $H_1$ containing $a$ and in which $a$ is primitive.
Note for future use that witout loss of generality, we can assume that $H'$ is a free group of rank at least $3$.
Let $\cala_1\subset \cala$ be the finite index subgroup made up of automorphisms that preserve $H'$.
Then $\cala_1\subset \cala$ still acts on $T$, and let
let $\cala'\subset \cala_1$ be the finite index subgroup acting trivially on the quotient graph $T/H'$.
Let $T_e$ be the tree obtained from $T$ by collapsing every  edge  that is not in the $H'$--orbit of $e$.
Then the action $H'\actson T_e$ is dual to a one-edge splitting of $H'$ whose edge stabilizer is generated by the primitive element $a\in H'$,
and this splitting is $\cala'$--invariant.

We now construct a finite index subgroup $H''\subset H'$ such that
the first Betti number of $T_e/H''$ is at least $2$.
Write the splitting of $H'$ dual to its action on $T_e$ as the amalgam $H'=A*_{\grp{a}} B$ or as the HNN extension $H'=C*_{\grp{a}}$.
Then $A$ and $B$ have rank at least 2. Since the rank of $H_1$ is at least 3, so is the rank of $H'$, and $C$ has rank at least $3$.
It follows that there exists a epimorphism $\pi:H'\onto \bbZ/3\bbZ$ that kills $a$ and is such that
the restriction of $\pi$ to each factor $A$, $B$ or $C$ is onto.
Now take $H'' = \ker \pi$, and $\Gamma''=T_e/H''$.
Since $\pi(a)=1$, $\Gamma''$ has 3 edges, and since the restriction of $\pi$ to each factor is onto,
$\Gamma''$ has the same number of vertices as $T_e / H'$ (\ie either 1 or 2).
Thus $\Gamma''$ has first Betti number at least 2.
Now the finite index subgroup $\cala''\subset\cala'$ of automorphisms preserving $H''$ and acting trivially on $T_e/H''$
maps onto $\bbF_2$.

Finally, we now treat the case where the JSJ decomposition $\Gamma$ is   a trivial decomposition consisting of a QH vertex
corresponding to a punctured Klein bottle.
By \cite[Prop. A.3]{Stukow},
the punctured Klein bottle has a unique isotopy class of essential 2-sided simple closed curves which does not bound a Mobius band.
Thus, the splitting $\Gamma_0$ over this curve is invariant under $\cala$, and one can argue using $\Gamma_0$ as we did with the JSJ decomposition $\Gamma$.
\end{proof}

We can now give a proof of Corollary \ref{corspecial:mccool}, which says that any vastness property $\calp$ that holds for $\Out(\bbF_n)$, for $n\geq 2$, and for all non-virtually abelian mapping class groups of punctured hyperbolic
surfaces
will hold for any McCool group that is not virtually abelian.

\begin{proof}[Proof of Corollary \ref{corspecial:mccool}]
  If $M=\Mc(\bbF,\calc)$ does not satisfy $\calp$, then
the only  possibilities in Theorem \ref{thm_MC}
are the assertion \ref{it_sporadic} with $M$ virtually abelian,
or  the assertion \ref{it_1bout}
with each $\MCG(S_j)$ virtually abelian.
It follows that $M_0$ is a virtually polycyclic subgroup of $\Out(\bbF)$, hence virtually abelian by \cite{BFH_solvable}.
\end{proof}

\bibliographystyle{alpha}
\bibliography{bibliography_RAAG3}


\vspace{8mm}
\begin{minipage}{0.5\textwidth}{\small
 \begin{flushleft}
 Vincent Guirardel\\
 Institut de Recherche Math\'ematique de Rennes\\
 Universit\'e de Rennes 1 et CNRS (UMR 6625)\\
 263 avenue du G\'en\'eral Leclerc, CS 74205\\
 F-35042  RENNES C\'edex\\
 \emph{e-mail:} \texttt{vincent.guirardel@univ-rennes1.fr}\\[8mm]
 \end{flushleft}}
\end{minipage}\hspace{0.05\textwidth}
\begin{minipage}{0.45\textwidth}{\small
 \begin{flushleft}
 Andrew Sale\\
 Department of Mathematics \\ 310 Malott Hall \\ Cornell University \\ Ithaca, NY USA 14853 \\
 \emph{e-mail:} \texttt{andrew.sale@some.oxon.org}\\[8mm]
 \end{flushleft}}
 \end{minipage}

\end{document}